\documentclass[12pt,a4paper,openany]{book}
\usepackage{amsmath}
\usepackage{amsmath}
\renewcommand\baselinestretch{1.2}
\usepackage{CJK,CJKnumb}
\usepackage[all]{xy}

\usepackage{amsmath,amsthm,amssymb}
\usepackage{amssymb}
\usepackage{amsfonts,amssymb,amsthm}

\usepackage{makeidx}
\makeindex

\newtheorem{thm}{Theorem}[section]
\newtheorem{lem}[thm]{Lemma}
\newtheorem{cor}[thm]{Corollary}
\newtheorem{prop}[thm]{Proposition}

\newtheorem{defn}[thm]{Definition}

\newtheorem{rem}[thm]{Remark}
\numberwithin{equation}{section}

\begin{document}
\begin{CJK*}{GBK}{song}
\CJKtilde

\begin{titlepage}
~~~ \vspace{2cm}

 {\bf \LARGE
\begin{center}
Gorenstein Homological Algebra of Artin Algebras
\end{center}}\vspace{3.cm}

\centerline{\bf \huge Xiao-Wu Chen}\vspace{8cm}

{\large \bf
\begin{center}
Department of Mathematics\\University of Science and Technology of China \\
Hefei, 230026,  People's Republic of China
\end{center}}

{\large \bf
\begin{center} { March  \ \ 2010}
\end{center}}

\end{titlepage}

\newpage

 \thispagestyle{empty}
 {\LARGE \bf
\begin{center} Acknowledgements \end{center}}\vspace{0.3cm}

I would like to thank my postdoctoral  mentor Professor Sen Hu for
his support. I also would like to thank Professor Pu Zhang,
Professor Henning Krause and Professor Yu Ye for their help.

I benefit from  private communications with Professor Apostolos
Beligiannis and Professor Edgar Enochs. I am indebted to Dr. Nan Gao
and Dr. Guodong Zhou for their encouragement.

During my postdoctoral research, I am supported by two grants from
China Postdoctoral Science Foundation and a grant from K.C. Wong
Education Foundation, Hong Kong. I am also partly supported by
Alexander von Humboldt Stiftung and National Natural Science
Foundation.

The last but not the least, I would like to thank my wife Jue for
her love and support.

\newpage
\pagenumbering{roman}

\addcontentsline{toc}{chapter}{\numberline {}{\bf Abstract }}
{\LARGE \bf \begin{center} Abstract\end{center}}\vspace{0.3cm}

Gorenstein homological algebra is a kind of relative homological
algebra which has been developed to a high level since more than
four decades.

In this report we review the basic theory of Gorenstein homological
algebra of artin algebras. It is hoped that such a theory will help
to understand the famous Gorenstein symmetric conjecture of artin
algebras.

With only few exceptions all the results in this report are
contained in the existing literature. We have tried to keep the
exposition as self-contained as possible. This report can be viewed
as a preparation for learning the newly developed theory of
virtually Gorenstein algebras.

In Chapter 2 we recall the basic notions in Gorenstein homological
algebra with particular emphasis on finitely generated
Gorenstein-projective modules, Gorenstein algebras and CM-finite
algebras.

In Chapter 3 based on a theorem by Beligiannis we study the
Gorenstein-projective resolutions and various Gorenstein dimensions;
we also discuss briefly Gorenstein derived categories in the sense
of Gao and Zhang.

We include three appendixes: Appendix A treats cotorsion pairs;
Appendix B sketches a proof of the theorem by Beligiannis; Appendix
C provides a list of open problems in Gorenstein homological algebra
of artin algebras.

\vskip 30pt


%
%

\end{CJK*}

\setlength{\topskip}{-35mm} \tableofcontents

\newpage

\thispagestyle{empty} ~~~

\newpage

\renewcommand\baselinestretch{1.2}

\pagenumbering{arabic}

\chapter{Introduction}

The main idea of homological algebra in module categories is to
replace a module by its projective (or injective) resolution. In
this way one defines derived functors of a given functor, which
provide more information on the given functor. Roughly speaking, the
information obtained measures how far the given functor is from
being exact. Here the notions of projective module and projective
resolution play a central role.

Relative homological algebra is initiated by Eilenberg and Moore
(\cite{EM}). The idea of relative homological algebra is that one
might replace projective modules by certain classes of modules and
then ``pretends"  that these modules are projective. Let us call
these modules temporarily relatively projective.  Suppose that one
has a resolution of any module with respect to these relatively
projective modules. Then for a given functor one defines certain
derived functors via the resolution. These derived functors will be
the hero in the theory of relative homological algebra. Evidently
the choice of relatively projective modules will be vital in
relative homological algebra. As a matter of fact, a different
choice of such modules will lead to a different theory.

There is another point of view on relative homological algebra. For
the chosen class of relatively projective modules, one can associate a
class of short exact sequences on which theses chosen modules behave
like projective modules. Such a class of short exact sequences
provides a new exact structure on the module category and then one
gets a new exact category in the sense of Quillen (\cite{Qui}). Then
it follows that relative homological algebra is just homological
algebra of certain exact categories. Again these exact categories depend on
the choice of these relatively projective modules.

Gorenstein homological algebra is a kind of relative homological
algebra, where the relatively projective modules are chosen to be
Gorenstein-projective modules. Finitely generated
Gorenstein-projective modules over a noetherian ring are introduced
by Auslander and Bridger under the name ``modules of G-dimension
zero" (\cite{ABr}). Over a commutative Gorenstein ring these modules
are equal to the maximal Cohen-Macaulay modules. Auslander and
Bridger introduce the notion of G-dimension for a finitely generated
module and then they generalize the famous Auslander-Buchbaum
formula with projective dimension replaced by the G-dimension. The
notion of arbitrary Gorenstein-projective modules over an arbitrary
ring is invented by Enochs and Jenda (\cite{EJ1}). Later the theory
of Gorenstein-projective modules is  studied intensively by Enochs's
school and others. Gorenstein derived functors are then defined
using a Gorenstein-projective resolution of a module
(\cite{EJ,Hol2}). However it is not a priori that such a resolution
exists for an arbitrary module. A recent and remarkable result due
to J{\o}rgensen states that for a large class of rings such a
resolution always exists (\cite{Jor2}). Inspired by these results
Gao and Zhang introduce the notion of Gorenstein derived category
(\cite{GZ}; also see \cite{Ch3}), which is a category with a higher
structure in the theory of Gorenstein homological algebra.

Dual to Gorenstein-projective modules one has the notion of
Gorenstein-injective module. These modules play the role of
injective modules in the classical homological algebra. Using
Gorenstein-injective modules one can define the Gorenstein-injective
coresolutions of modules and then define the corresponding derived
functors for a given functor. However, in general it is not clear
how these derived functors are related to the ones given by
Gorenstein-projective resolutions.

In this report, we study the Gorenstein homological algebra of artin
algebras. The restriction to artin algebras is mainly because of a
matter of taste. Due to a work by Auslander and Reiten
Gorenstein-projective modules are closely related to the famous
Gorenstein symmetric conjecture in the theory of artin algebras
(\cite{AR, AR1991}; also see \cite{Hap2}). We hope that an intensive study
of Gorenstein homological algebra of artin algebras will help to
understand this conjecture.

This report is organized as follows.

In Chapter 2, we provide some preliminaries on Gorenstein
homological algebra: we treat the category of finitely generated
Gorenstein-projective modules in detail; we also discuss other
classes of modules which are important in Gorenstein homological
algebra; we briefly discuss Gorenstein algebras, CM-finite algebras,
CM-free algebras and virtually Gorenstein algebras.

Chapter 3 treats the main topic in Gorenstein homological algebra:
we study the Gorenstein-projective extension groups in detail; we
study various Gorenstein dimensions of modules and algebras and
study the class of modules having finite Gorenstein dimension; we
briefly discuss Gorenstein derived categories.

We include three appendixes: Appendix A treats cotorsion pairs and
related notions; Appendix B sketches a proof of an important theorem
due to Beligiannis; Appendix C collects some open problems, most of
 which are related to CM-finite algebras.

Let us finally point out that with only few exceptions the results
in this report are contained in the existing literature.  This
report may be viewed as a preparation for reading the beautiful
theory of virtually Gorenstein algebras developed in \cite{BR, Bel2,
Bel3, BK}.

\chapter{Preliminaries}

\section{Gorenstein-Projective Modules}

In this section we study for an artin algebra the category of
finitely generated Gorenstein-projective modules. Such modules are
also known as maximal Cohen-Macaulay modules. These modules play a
central role in the theory of Gorenstein homological algebra.

Throughout $A$ will be an artin $R$-algebra where $R$ is a
commutative artinian ring. Denote by $A\mbox{-mod}$ the category of
finitely generated left $A$-modules. In this section all modules are
considered to be finitely generated. A left $A$-module $X$ is often
written as $_AX$ and a right $A$-module $Y$ is written as $Y_A$.
Right $A$-modules are viewed as left $A^{\rm op}$-modules. Here
$A^{\rm op}$ denotes the opposite algebra of $A$. In what follows,
$A$-modules always mean left $A$-modules.

 For an $A$-module $X$, write $DX={\rm Hom}_R(X, E)$ its Matlis dual where $E$ is
the minimal injective cogenerator for $R$. Note that $DX$ has a
natural right $A$-module structure and then it is viewed as an
$A^{\rm op}$-module.

 A complex $C^\bullet=(C^n, d^n)_{n\in \mathbb{Z}}$ of $A$-modules  consists
 of a family $\{C^n\}_{n\in \mathbb{Z}}$ of $A$-modules  and
 differentials $d^n\colon C^n\rightarrow C^{n+1}$ satisfying
$d^{n}\circ d^{n-1}=0$. Sometimes a complex is written as a sequence of $A$-modules
$\cdots \rightarrow C^{n-1}\stackrel{d^{n-1}}\rightarrow C^{n}\stackrel{d^{n}}\rightarrow C^{n+1}\rightarrow \cdots$.
For each $n\in \mathbb{Z}$ denote by $B^n(C^\bullet)={\rm Im}\; d^{n-1}$ and $Z^n(C^\bullet)={\rm Ker}\; d^n$
the $n$-th coboundary and cocycle of $C^\bullet$, respectively. Note that $B^n(C^\bullet)\subseteq Z^n(C^\bullet)$.
Denote by $H^n(C^\bullet)=Z^n(C^\bullet)/B^n(C^\bullet)$ the $n$-th cohomology of the complex $C^\bullet$.

 A complex $C^\bullet$ of $A$-modules is \emph{acyclic}\index{complex!acyclic} provided that
  it is exact as a sequence, or equivalently, $H^n(C^\bullet)=0$ for all $n$.
   Following \cite[p.400]{AM} a complex $P^\bullet$ of projective $A$-modules is said to be \emph{totally
acyclic}\index{complex!totally acyclic} provided it is acyclic and
the Hom complex ${\rm Hom}_A(P^\bullet, A)$ is also acyclic.

Following Enochs and Jenda \cite{EJ1,EJ} we have the following
definition.

\begin{defn} \label{defn:Gp}
An $A$-module $M$ is said to be (finitely generated)
\emph{Gorenstein-projective}\index{Gorenstein-projective module!finitely generated} provided that there is a totally
acyclic complex $P^\bullet$ of projective modules such that its
$0$-th cocycle $Z^0(P^\bullet)$ is isomorphic to $M$.
\end{defn}

We will denote by $A\mbox{-Gproj}$ the full subcategory of
$A\mbox{-mod}$ consisting of Gorenstein-projective modules.

In Definition \ref{defn:Gp} the complex $P^\bullet$ is said to be a
\emph{complete resolution}\index{resolution!complete} of $M$.
Observe that each cocycle in a totally acyclic complex is
Gorenstein-projective. Note that any projective module $P$
  is Gorenstein-projective, since its complete resolution can be taken as $\cdots \rightarrow 0\rightarrow P\stackrel{{\rm Id}_P}\rightarrow P \rightarrow 0\rightarrow \cdots$. Therefore, we have $A\mbox{-proj}\subseteq A\mbox{-Gproj}$.

For an $A$-module $M$ write  $M^*={\rm Hom}_A(M, A)$ which has a
canonical right $A$-module structure. This gives rise to a
contravariant functor $(-)^*\colon A\mbox{-mod}\rightarrow A^{\rm
op}\mbox{-mod}$. For a complex $M^\bullet$ we denote by
$(M^\bullet)^*$ the Hom complex ${\rm Hom}_A(M^\bullet, A)$.

Recall that for an $A$-module $M$, one has the following
\emph{evaluation morphism}\index{evaluation morphism}
$${\rm ev}_M\colon M\longrightarrow M^{**}={\rm Hom}_{A^{\rm op}}({\rm Hom}_A(M, A), A)$$
such that ${\rm ev}_M(m)(f)=f(m)$;  $M$ is called \emph{reflexive
}\index{module!reflexive}if ${\rm ev}_M$ is an isomorphism. For example, (finitely
generated) projective modules are reflexive.

Denote by $^\perp A$ the full subcategory of $A\mbox{-mod}$
consisting of modules $M$ with the property ${\rm Ext}_A^i(M, A)=0$
for $i\geq 1$. By a dimension-shift argument, one obtains that ${\rm
Ext}^i_A(M, L)=0$ for all $i\geq 1$, $M\in {^\perp A}$ and $L$
 having finite projective dimension.

\begin{lem}\label{lem:totallyacyclic}
Let $P^\bullet$ be a complex of projective $A$-modules. Then the
following statements are equivalent:
\begin{enumerate}
\item[(1)] the complex $P^\bullet$ is totally acyclic;
\item[(2)] the complex $P^\bullet$ is acyclic and each  cocycle $Z^i(P^\bullet)$
lies in $^\perp A$;
\item[(3)] the complex $(P^\bullet)^*$ is totally acyclic.
\end{enumerate}
\end{lem}

\begin{proof}
Note that for a complex $P^\bullet$ of projective modules the
evaluation morphisms induce an isomorphism $P^\bullet \simeq
(P^\bullet)^{**}$ of complexes. Then the equivalence
$(1)\Leftrightarrow (3)$ follows from the definition. The
equivalence $(1) \Leftrightarrow (2)$ follows from the following
observation: for an acyclic complex $P^\bullet$ of projective
modules and for each $i\in \mathbb{Z}$, the truncated complex
$\cdots \rightarrow P^{i-2}\rightarrow P^{i-1}\rightarrow 0$ is a
projective resolution of the cocycle $Z^i(P^\bullet)$ and
 then we have  $H^{-i+k+1}((P^\bullet)^*)\simeq {\rm Ext}_A^k(Z^i(P^\bullet), A)$
  for all $k\geq 1$.
\end{proof}

\begin{cor}\label{cor:infinitedimension}
We have that $A\mbox{-{\rm Gproj}}\subseteq {^\perp A}$. Then for a
Gorenstein-projective module  $M$ we have:
\begin{enumerate}
\item[(1)] ${\rm Ext}_A^i(M, L)=0={\rm Tor}_i^A(L', M)$ for all $i\geq 1$,
 $_AL$ of finite projective dimension and $L'_A$ of finite
injective dimension;
\item[(2)] $M$ is either projective or has
infinite projective dimension.
\end{enumerate}
\end{cor}

\begin{proof}
Note that $D{\rm Tor}_i^A(L', M)\simeq {\rm Ext}_A^i(M, DL')$. Then
the first statement follows from Lemma \ref{lem:totallyacyclic}(2).
For the second, we apply the first statement. Then it follows from
the fact that a module $X$ of finite projective dimension $d$
satisfies that ${\rm Ext}_A^d(X, A)\neq 0$.
\end{proof}

 \begin{lem}\label{lem:Gproj}
 Let $M\in A\mbox{-{\rm mod}}$. Then the  following statements are equivalent:
 \begin{enumerate}
 \item[(1)] $M$ is Gorenstein-projective;
 \item[(2)] there exists a long exact sequence $0\rightarrow M \rightarrow P^0
          \rightarrow P^1 \rightarrow P^2 \rightarrow \cdots$ with each $P^i$ projective
          and each cocycle in $^\perp A$;
 \item[(3)] $M\in {^\perp A}$, $M^*\in {^\perp}(A_A)$ and $M$ is reflexive.
 \end{enumerate}
 \end{lem}

\begin{proof}
 The implication ``$(1)\Rightarrow (2)$" follows from Lemma
 \ref{lem:totallyacyclic}(2).  For the other direction, assume (2) and
 take a projective resolution $\cdots \rightarrow P^{-2} \rightarrow
P^{-1}\rightarrow M \rightarrow
 0$. By splicing we get an acyclic complex $P^\bullet$ such that $Z^0(P^\bullet)\simeq
 M$. Note that since $^\perp A \subseteq A\mbox{-mod}$ is closed
 under taking kernels of epimorphisms. It follows that all the cocycles in $P^\bullet$
 lie in $^\perp A$. By Lemma \ref{lem:totallyacyclic}(2) the complex
 $P^\bullet$ is totally acyclic. We are done.

   To see ``$(1) \Rightarrow (3)$", first note that $M\in {^\perp A}$; see
   Corollary \ref{cor:infinitedimension}. For others, take a
   complete resolution $P^\bullet$  of $M$. Note that $(P^\bullet)^*$ is totally acyclic and that
   $Z^i((P^\bullet)^*)=(Z^{-i+1}(P^\bullet))^*$. It follows that
   $M^*$ is Gorenstein-projective and then $M^* \in {^\perp
   A_A}$. For the same reason we have $Z^i((P^\bullet)^{**})=(Z^{-i+1}((P^\bullet)^*))^*
   = Z^i(P^\bullet)^{**}$. Note that evaluation morphisms induce an isomorphism $P^\bullet \simeq
   (P^\bullet)^{**}$ of complexes. Then it follows that $M$ is reflexive.

For ``$(3)\Rightarrow (1)$", take projective resolutions $\cdots
\rightarrow P^{-3}\rightarrow P^{-2}\rightarrow P^{-1}\rightarrow
M\rightarrow 0$ and $\cdots \rightarrow Q^{-2}\rightarrow
Q^{-1}\rightarrow Q^0\rightarrow M^*\rightarrow 0$.  Apply the
functor $(-)^*$ to the second resolution. By $(3)$ the resulting
complex $0\rightarrow (M^*)^*\rightarrow (Q^0)^*\rightarrow
(Q^{-1})^*\rightarrow (Q^{-2})^* \rightarrow \cdots$ is acyclic.
Note that $M$ is reflexive. Then by splicing the first and the third
complexes together we obtain a complete resolution of $M$.
\end{proof}

\begin{rem}
In view of Lemma \ref{lem:Gproj}(3) Gorenstein-projective modules
are the same as \emph{totally reflexive modules} in \cite[section
2]{AM}. Sometimes they are also called \emph{modules of G-dimension
zero} \index{module!of G-dimension zero}(\cite{Au,ABr}) or
\emph{maximal Cohen-Macaulay modules} (\cite[Definition
3.2]{Bel2})\index{module!maximal Cohen-Macaulay}. In view of Lemma
\ref{lem:Gproj}(2) we note that the subcategory $A\mbox{-{\rm
Gproj}}$ is a special case of the categories studied in
\cite[Proposition 5.1]{AR}.
\end{rem}

The following is an immediate  consequence of Lemma
\ref{lem:Gproj}(3).

\begin{cor}\label{cor:duality}
There is a duality $(-)^*\colon A\mbox{-{\rm
Gproj}}\stackrel{\sim}\longrightarrow A^{\rm op}\mbox{-{\rm Gproj}}$
with its quasi-inverse given by $(-)^*={\rm Hom}_{A^{\rm op}}(-,
A)$. \hfill $\square$
\end{cor}

Recall that a full additive subcategory $\mathcal{X}$ of
$A\mbox{-mod}$ is \emph{resolving} provided that it contains all
projective modules and is closed under extensions, taking kernels of
epimorphisms and direct summands (\cite{ABr})\index{subcategory!resolving}. For example,
$A\mbox{-proj}\subseteq A\mbox{-mod}$ is resolving. We will see
shortly that $A\mbox{-Gproj}\subseteq A\mbox{-mod}$ is resolving.

The following result collects some important properties of
Gorenstein-projective modules. (1)-(3) are due to \cite[Proposition
5.1]{AR} (compare \cite[Lemma 2.3]{AM} and \cite[Theorem 2.5]{Hol})
and (4) is \cite[Corollary 2.11]{Hol}.

\begin{prop}\label{prop:Gproj}
Let $\xi\colon 0\rightarrow N\stackrel{f}\rightarrow
M\stackrel{g}\rightarrow L\rightarrow 0$ be a short exact sequence
of $A$-modules. Then we have the following statements:
\begin{enumerate}
\item[(1)] if $N, L$ are Gorenstein-projective, then so is $M$;
\item[(2)] if $\xi$ splits and $M$ is Gorenstein-projective, then so are
$N, L$;
\item[(3)] if $M, L$ are Gorenstein-projective, then so is $N$;
\item[(4)] if ${\rm Ext}^1_A(L, A)=0$ and $N, M$ are
Gorenstein-projective, then so is $L$.
\end{enumerate}
\end{prop}

\begin{proof}
(1). Since $N$ and $L$ are Gorenstein-projective, we may take
monomorphisms $N\stackrel{i_N}\rightarrow P$ and
$L\stackrel{i_L}\rightarrow Q$ such that $P$ and $Q$ are projective
and the cokernels $N^1$ and $L^1$ of $i_N$ and $i_L$, respectively,
are Gorenstein-projective. Since ${\rm Ext}^1_A(L, P)=0$, from the
long exact sequence obtained by applying the functor ${\rm Hom}_A(-,
P)$ to $\xi$ we infer that the induced map ${\rm Hom}_A(M,
P)\rightarrow {\rm Hom}_A(N, P)$ is surjective. In particular, there
is a morphism $a\colon M\rightarrow P$ such that $a\circ f=i_N$.
Therefore we have the following exact diagram
\[\xymatrix{
0\ar[r] & N  \ar[d]^{i_N} \ar[r]^-{f} & M \ar[d]^{\binom{a}{i_L\circ
g}} \ar[r]^-{g} & L \ar[d]^{i_L} \ar[r] & 0\\
0\ar[r] & P \ar[r]^-{\binom{1}{0}} & P\oplus Q \ar[r]^-{(0\; 1)} & Q
\ar[r] & 0. }\]
 By Snake Lemma the middle column map is monic and
there is an induced short exact sequence $0\rightarrow
N^1\rightarrow M^1\rightarrow L^1\rightarrow 0$ where $M^1$ is the
cokernel of the middle column map. Note that $N^1, L^1$ are
Gorenstein projective. In particular, $N^1, L^1\in {^\perp A}$ and
then we have $M^1\in {^\perp A}$. Iterating this argument, using
Lemma \ref{lem:Gproj}(2) we show that $M$ is Gorenstein-projective.

The statement (2) amounts to the fact that a direct summand of a
Gorenstein-projective module is again Gorenstein-projective. For
this end, let $N\oplus L$ be Gorenstein-projective. Note that
$N\oplus L\in {^\perp A}$ and then $N\in {^\perp A}$. Take a short
exact sequence $0\rightarrow N\oplus L \rightarrow P\rightarrow
G\rightarrow 0$ such that $P$ is projective and $G$ is
Gorenstein-projective. Then the cokernel $N^1$ of the monomorphism
$N\rightarrow P$ fits into a short exact sequence $\eta\colon
0\rightarrow L\rightarrow N^1\rightarrow G\rightarrow 0$. We add the
trivial exact sequence $0\rightarrow N\stackrel{{\rm
Id}_N}\rightarrow N \rightarrow 0\rightarrow 0$ to $\eta$. Note that
both $L\oplus N$ and $G$ are Gorenstein-projective. By (1)  we infer
that $N^1\oplus N$ is Gorenstein-projective. Note that we have the
short exact sequence $0\rightarrow N\rightarrow P\rightarrow
N^1\rightarrow 0$ and that $N^1\in {^\perp A}$. We repeat the
argument by replacing $N$ with $N^1$ to get $N^2$ and a short exact
sequence $0\rightarrow N^1\rightarrow P^1\rightarrow N^2\rightarrow
0$. Continue this argument. Then we get a required long exact
sequence in Lemma \ref{lem:Gproj}(2), proving that $N$ is
Gorenstein-projective. Similarly $L$ is Gorenstein-projective.

For the statement (3), take a short exact sequence $0\rightarrow
L'\rightarrow P\rightarrow L\rightarrow 0$ such that $P$ is
projective and $L'$ is Gorenstein-projective. Consider the following
pullback diagram.
\[\xymatrix{ &  & 0\ar[d] & 0\ar[d]\\
            &   & L'\ar@{.>}[d] \ar@{=}[r] & L'\ar[d] \\
0\ar[r]  & N\ar@{=}[d] \ar@{.>}[r] & E\ar@{.>}[d] \ar@{.>}[r] &
P\ar[d] \ar[r] &
0\\
0\ar[r] & N \ar[r] & M\ar[d] \ar[r] & L\ar[d] \ar[r] & 0\\
& & 0 & 0 }\] Consider the short exact sequence in the middle
column. By (1) $E$ is Gorenstein-projective. Note that the short
exact sequence in the middle row splits since $P$ is projective.
Hence $E\simeq N\oplus P$. By (2) $N$ is Gorenstein-projective.

For the statement (4), take a short exact sequence $0\rightarrow
N\rightarrow P\rightarrow N'\rightarrow 0$ such that $P$ is
projective and $N'$ is Gorenstein-projective. Consider the following
pushout diagram.
\[\xymatrix{
& 0\ar[d] & 0\ar[d] \\
0\ar[r] & N \ar[d] \ar[r] & M \ar@{.>}[d] \ar[r] & L \ar@{=}[d]
\ar[r] & 0\\
0\ar[r] & P \ar@{.>}[r] \ar[d] & E \ar@{.>}[r] \ar@{.>}[d] & L
\ar[r] & 0\\
& N' \ar[d]\ar@{=}[r] & N' \ar[d]\\
& 0 & 0 }\] Consider the short exact sequence in the middle column.
By (1) we infer that $E$ is Gorenstein-projective. By the assumption
the short exact sequence in the middle row splits and then $E\simeq
P\oplus L$. Now applying (2) we are done.
\end{proof}

We denote by $A\mbox{-\underline{mod}}$ the \emph{stable
category}\index{category!stable} of $A\mbox{-mod}$ modulo projective
modules: the objects are the same as $A\mbox{-mod}$ while the
morphism space between two objects $M$ and $N$, denote by
${\underline{\rm Hom}}_A(M, N)$, is by definition the quotient
$R$-module ${\rm Hom}_A(M, N)/P(M, N)$ where $P(M, N)$ is the
$R$-submodule of ${\rm Hom}_A(M, N)$ consisting of morphisms
factoring through projective modules. The stable category
$A\mbox{-\underline{mod}}$ is additive and projective modules are
zero objects; for details, see \cite[p.104]{ARS}. Moreover, two
modules $M$ and $N$ become isomorphic in $A\mbox{-\underline{mod}}$
if and only if there exist projective modules $P$ and $Q$ such that
$M\oplus P\simeq  N\oplus Q$; compare \cite[Proposition 1.41]{ABr}.

For a module $M$ take a short exact sequence $0\rightarrow \Omega M
\rightarrow P\rightarrow M\rightarrow 0$ with $P$ projective. The
module $\Omega M$ is called a \emph{syzygy
module}\index{module!syzygy} of $M$. Note that syzygy modules of $M$
are not uniquely determined, while they are naturally isomorphic to
each other in $A\mbox{-\underline{mod}}$. In this sense we say that
$\Omega M$ is ``the" syzygy module of $M$. Moreover, we get the
\emph{syzygy functor}\index{functor!syzygy} $\Omega\colon
A\mbox{-\underline{mod}}\rightarrow A\mbox{-\underline{mod}}$. For
each $i\geq 1$, denote by $\Omega^{i}$ the $i$-th power of $\Omega$
and then for a module $M$, $\Omega^i M$ is the $i$-th \emph{syzygy
module} of $M$; for details, see \cite[p.124]{ARS}.

Recall the following basic property of these syzygy modules\footnote{The author thanks Rene Marczinzik for pointing out an error in the previous version.}.

\begin{lem} \label{lem:basicpropertyofsyzygy}
Let $M, N$ be $A$-modules and let $k\geq 1$.  Then there exists a
natural epimorphism
 $$\underline{\rm
Hom}_A(\Omega^k M, N) \twoheadrightarrow {\rm Ext}_A^k(M, N).$$ 
If in addition ${\rm Ext}_A^{i}(M, A)=0$ for $1\leq i\leq k$, then the above map is an isomorphism. \hfill $\square$
\end{lem}

The second part of the following observation seems  of interest.

\begin{cor}\label{cor:syzygy} Let $M$ be an $A$-module and let $d\geq 1$.
Then we have the following statements:
\begin{enumerate}
\item[(1)] if $M$ is a Gorenstein-projective module, then so are $\Omega^i M$ for $i\geq
1$;
\item[(2)] if ${\rm Ext}_A^i(M, A)=0$ for $1\leq i\leq d$ and $\Omega^d M$ is
Gorenstein-projective,  then so is $M$.
\end{enumerate}\end{cor}

\begin{proof}
The first statement follows by  applying Proposition
\ref{prop:Gproj}(3) repeatedly. Just consider the long exact
sequence $\eta\colon 0\rightarrow \Omega^d M \rightarrow
P^{d-1}\rightarrow \cdots \rightarrow P^0 \rightarrow M\rightarrow
0$. For the second one, note that  from the assumption and a
dimension-shift argument we have that ${\rm Ext}^1_A(-, A)$ vanishes
on all the cocycles of $\eta$. Then the second
statement follows by applying Proposition \ref{prop:Gproj}(4)
repeatedly.
\end{proof}

Recall  the construction of the \emph{transpose} ${\rm Tr}M$ of a
module  $M$: take a projective presentation $P^{-1}\rightarrow
P^0\rightarrow M\rightarrow 0$ and then define the right $A$-module
${\rm Tr}M$ to be the cokernel of the morphism $(P^0)^*\rightarrow
(P^{-1})^*$. Again the module ${\rm Tr}M$ is not uniquely determined,
while it is unique when viewed as an object in
$A\mbox{-\underline{mod}}$. This defines the \emph{transpose
functor} \index{functor!transpose} ${\rm Tr}\colon
A\mbox{-\underline{mod}}\rightarrow A^{\rm
op}\mbox{-\underline{mod}}$ which is  contravariant; it is even a
duality of categories. Observe that there is a natural isomorphism
$\Omega^2 M \simeq ({\rm Tr}M)^*$ (certainly in the stable category
$A\mbox{-\underline{mod}}$). For details, see \cite{ABr} and
\cite[p.105]{ARS}.

The following result is interesting.

\begin{prop}\label{prop:transpose}
Let $M$ be an $A$-module. Then $M$ is Gorenstein-projective if and
only if ${\rm Tr}M$ is Gorenstein-projective. Moreover, there is an
isomorphism ${\rm Tr}M\simeq (\Omega^2 M)^*$ in
$A\mbox{-\underline{\rm mod}}$ which is functorial in $M\in
A\mbox{-{\rm Gproj}}$.
\end{prop}

\begin{proof}
For the ``only if" part of the first statement, assume that $M$ is
Gorenstein-projective. Take a complete resolution $P^\bullet$ of
$M$. By definition ${\rm Tr}M$ is isomorphic to $Z^3((P^\bullet)^*)$
(in the stable category). Then ${\rm Tr}M$ is Gorenstein-projective,
since the complex $(P^\bullet)^*$ is totally acyclic; see Lemma
\ref{lem:totallyacyclic}(3).

To see the ``if" part, first note the following exact sequence
(\cite[Chapter IV, Proposition 3.2]{ARS})
$$0\longrightarrow {\rm Ext}^1_{A^{\rm op}}({\rm Tr}M, A) \longrightarrow M \stackrel{{\rm ev}_M}\longrightarrow
M^{**}\longrightarrow {\rm Ext}^2_{A^{\rm op}}({\rm Tr}M,
A)\longrightarrow 0.$$ Since ${\rm Tr}M$ is Gorenstein-projective
the two end terms vanish; see Corollary
\ref{cor:infinitedimension}(1). Then $M$ is reflexive. Take a
projective presentation $P^{-1}\rightarrow P^0\rightarrow
M\rightarrow 0$. Then we have an exact sequence $0\rightarrow
M^*\rightarrow (P^{-1})^*\rightarrow (P^0)^*\rightarrow {\rm
Tr}M\rightarrow 0$. Applying Proposition \ref{prop:Gproj}(3) twice
we obtain that $M^*$ is Gorenstein-projective. By Corollary
\ref{cor:duality} $M\simeq (M^*)^*$ is Gorenstein-projective.

Note that $\Omega^2 M \simeq ({\rm Tr}M)^*$ and that ${\rm Tr}M$ is
reflexive. Then the second statement follows.
\end{proof}

Recall that an \emph{exact category} \index{category!exact}  in the sense of Quillen is an
additive category endowed with an \emph{exact structure}\index{exact structure}, that is, a
distinguished class of ker-coker sequences which are called
\emph{conflations}\index{conflation}, subject to certain axioms (\cite[Appendix
A]{Kel}). For example, an extension-closed subcategory of an abelian
category is an exact category such that conflations are short exact
sequences with terms in the subcategory.

Recall that an exact category $\mathcal{A}$ is \emph{Frobenius} \index{category!exact!Frobenius}
provided that it has enough projective and enough injective objects
and the class of projective objects coincides with the class of
injective objects. For a Frobenius exact category, denote by
$\underline{\mathcal{A}}$ its stable category modulo projective
objects; it is a \emph{triangulated category}\index{category!triangulated} such that its shift functor
is the quasi-inverse of the syzygy functor and triangles are induced
by conflations. For details, see \cite[Chapter I, section 2]{Hap}.

In what follows we denote by $A\mbox{-\underline{Gproj}}$ the full
subcategory of $A\mbox{-\underline{mod}}$ consisting of
Gorenstein-projective modules.

\begin{prop} \label{prop:Gprojtriangulated}
Let $A$ be an artin algebra. Then we have
\begin{enumerate}
\item[(1)] the category $A\mbox{-{\rm Gproj}}$ is a Frobenius exact
category, whose projective objects are equal to projective modules;
\item[(2)] the stable category $A\mbox{-\underline{{\rm Gproj}}}$
modulo projective modules is
triangulated.
\end{enumerate}
\end{prop}

\begin{proof}
By Proposition \ref{prop:Gproj}(1) $A\mbox{-{\rm Gproj}}\subseteq
A\mbox{-mod}$ is closed under extensions, and then $A\mbox{-Gproj}$
is an exact category. Note that $A\mbox{-Gproj}\subseteq {^\perp
A}$. Then projective modules are projective and injective in
$A\mbox{-Gproj}$. Then (1) follows, while (2) follows from (1).
\end{proof}

\begin{rem}\label{rem:shiftfunctor}
Note that $\Omega\colon A\mbox{-\underline{\rm Gproj}} \rightarrow
A\mbox{-\underline{\rm Gproj}}$ is invertible and its quasi-inverse
$\Sigma$ is  the shift functor for the triangulated category
A\mbox{-\underline{\rm Gproj}}; see \cite[p.13]{Hap}. Moreover, we
have a natural isomorphism $\Sigma M\simeq \Omega(M^*)^*$ for $M\in
A\mbox{-{\rm Gproj}}$.

In fact, applying $(-)^*$ to the short exact sequence $0\rightarrow
\Omega(M^*)\rightarrow P\rightarrow M^*\rightarrow 0$ with $P_A$
projective, we get an exact sequence $0\rightarrow M\rightarrow
P^*\rightarrow \Omega(M^*)^*\rightarrow 0 $ (here we use that $M$ is
reflexive and ${\rm Ext}_{A^{\rm op}}^1(M^*, A)=0$). Then we
conclude that $\Sigma M\simeq \Omega(M^*)^*$. Let us remark that one
can also infer this natural isomorphism from Corollary
\ref{cor:trianglequivalence} below.
\end{rem}

The following observation is of interest.

\begin{lem}\label{lem:syzygyofGproj}
Let $M$ be a non-projective indecomposable Gorenstein-projective
$A$-module. Consider its projective cover $\pi \colon
P(M)\rightarrow M$. Then ${\rm Ker}\; \pi$ is non-projective and
indecomposable.
\end{lem}

\begin{proof}
Note that ${\rm Ker}\; \pi$ is isomorphic to $\Omega M$ in
$A\mbox{-\underline{mod}}$. By Remark \ref{rem:shiftfunctor} we
infer that ${\rm Ker}\; \pi$ is indecomposable in
$A\mbox{-\underline{mod}}$. By Krull-Schmidt Theorem we have ${\rm
Ker}\; \pi\simeq N\oplus P$ such that $N$ is non-projective and
indecomposable and $P$ is projective. Since ${\rm Ext}_A^1(M, P)=0$,
then the composite inclusion $P\hookrightarrow {\rm Ker}\; \pi
\hookrightarrow P(M)$ splits. On this other hand, the morphism $\pi$
is a projective cover. This forces that $P$ is zero.
\end{proof}

Note that it follows from Corollary \ref{cor:infinitedimension} that
the duality functors $(-)^*\colon
A\mbox{-Gproj}\stackrel{\sim}\longrightarrow A^{\rm
op}\mbox{-Gproj}$ and $(-)^*\colon A^{\rm
op}\mbox{-Gproj}\stackrel{\sim}\longrightarrow A\mbox{-Gproj}$ are
exact; see Corollary \ref{cor:duality}. Moreover, they restrict to
the well-known duality $ A\mbox{-proj}\simeq A^{\rm
op}\mbox{-proj}$. Then the following result follows immediately
(consult \cite[p.23, Lemma]{Hap}).

\begin{cor}\label{cor:trianglequivalence}
There is a duality $(-)^*\colon A\mbox{-\underline{\rm
Gproj}}\stackrel{\sim}\longrightarrow A^{\rm
op}\mbox{-\underline{\rm Gproj}}$ of triangulated categories such
that its quasi-inverse is given by $(-)^*={\rm Hom}_{A^{\rm op}}(-,
A)$. \hfill $\square$
\end{cor}

The following observation is of independent interest. For the notion
of cohomological functor, we refer to \cite[p.4]{Hap}.

\begin{prop}\label{prop:cohomologicalfunctor}
Let $M$ be an $A$-module. Then the functors $\underline{{\rm
Hom}}_A(M, -)\colon A\mbox{-\underline{{\rm Gproj}}}\rightarrow
R\mbox{-{\rm mod}}$ and $\underline{{\rm Hom}}_A(-, M)\colon
A\mbox{-\underline{{\rm Gproj}}}\rightarrow R\mbox{-{\rm mod}}$ are
cohomological.
\end{prop}

\begin{proof}
We only show that the first functor is cohomological and the second
can be proved dually. Since triangles in
$A\mbox{-\underline{{\rm Gproj}}}$ are induced by short exact
sequences in $A\mbox{-\underline{\rm Gproj}}$, it suffices to show that for any short exact sequence
$0\rightarrow X \stackrel{f}\rightarrow Y \stackrel{g}\rightarrow Z
\rightarrow 0$ of Gorenstein-projective modules, the induced
sequence $\underline{\rm Hom}_A(M, X) \rightarrow \underline{\rm
Hom}_A(M, Y) \rightarrow \underline{\rm Hom}_A(M, Z)$ is exact in
the middle.

 This amounts to proving the following statement: given a morphism $\alpha\colon M
\rightarrow Y$ such that $g\circ \alpha$ factors though a projective
module, then there is a morphism $\beta\colon M \rightarrow X$ such
that $\alpha-f\circ \beta$ factors through a projective module.
Assume that there is a morphism $\pi\colon P\rightarrow Z$ with $P$
projective such that there is a morphism $t\colon M\rightarrow P$
satisfying $g\circ \alpha=\pi \circ t$. Since $P$ is projective, we
may lift $\pi$ along $g$ to a morphism $\pi'\colon P\rightarrow Y$, that
is, $g\circ \pi'=\pi$. Note that $g\circ (\alpha-\pi'\circ
t)=0$ and then we infer that  there exists a morphism $\beta\colon M \rightarrow X$
such that $\alpha-\pi'\circ t=f\circ \beta$. In particular,
$\alpha-f\circ \beta$ factors through the projective module $P$. We
are done.
\end{proof}

\section{Other Relevant Modules}

In this section we discuss other classes of  modules in Gorenstein
homological algebra: finitely generated Gorenstein-injective
modules, large Gorenstein-projective modules and large
Gorenstein-injective modules; here ``large" means ``not necessarily
finitely generated". Strongly Gorenstein-projective modules are also
briefly discussed.

Let $A$ be an artin algebra. Denote by $A\mbox{-inj}$ the full
subcategory of $A\mbox{-mod}$ consisting of finitely generated
injective $A$-modules. Consider the \emph{Nakayama functors}\index{functor!Nakayama}
$\nu=DA\otimes_A-\colon A\mbox{-mod} \rightarrow A\mbox{-mod}$ and
$\nu^{-}={\rm Hom}_A(DA, -) \colon A\mbox{-mod}\rightarrow
A\mbox{-mod}$. Note that $(\nu, \nu^{-1})$ is an adjoint pair. Then
for any module $_AM$ we have natural morphisms $\nu\nu^{-}
M\rightarrow M$ and $M\rightarrow \nu^{-}\nu M$. Moreover, for
projective modules $P$ we have $P\stackrel{\sim}\longrightarrow
\nu^{-}\nu P$; for injective modules $I$ we have $\nu\nu^{-}
I\stackrel{\sim}\longrightarrow I$. In this way we get mutually
inverse equivalences $\nu\colon
A\mbox{-proj}\stackrel{\sim}\longrightarrow A\mbox{-inj}$ and
$\nu^{-} \colon A\mbox{-inj}\stackrel{\sim}\longrightarrow
A\mbox{-proj}$.

Note that we have natural isomorphisms $\nu\simeq D\circ (-)^*$ and
$\nu^{-}\simeq (-)^*\circ D$ of functors. Hence a module $M$ is
reflexive if and only if the natural morphism $M\rightarrow
\nu^{-}\nu M$ is an isomorphism; while the natural morphism
$\nu\nu^{-} M\rightarrow M$ is an isomorphism if and only if the
right $A$-module $DM$ is reflexive.

An acyclic complex $I^\bullet$ of injective $A$-modules is said to
be \emph{cototally acyclic}\index{complex!cototally
acyclic}\footnote{The terminology, which certainly is not  standard,
is introduced to avoid possible confusion.} provided that the Hom
complex $\nu^{-} I^\bullet={\rm Hom}_A(DA, I^\bullet)$ is acyclic.
An $A$-module $N$ is said to be (finitely generated)
\emph{Gorenstein-injective}\index{Gorenstein-injective
module!finitely generated} provided that there is a cototally
acyclic complex $I^\bullet$ such that its zeroth coboundary
$B^0(I^\bullet)$ is isomorphic to $N$. Finitely generated
Gorenstein-injective modules are also known as \emph{maximal
co-Cohen-Macaulay modules} (\cite{BR, Bel2})\index{module!maximal
co-Cohen-Macaulay}.

We denote by $A\mbox{-Ginj}$ the full subcategory of $A\mbox{-mod}$
consisting of Gorenstein-injective modules. Observe that
$A\mbox{-inj}\subseteq A\mbox{-Ginj}$.

Denote by $(DA)^\perp$ the full subcategory of $A\mbox{-mod}$
consisting of modules $N$ with the property ${\rm Ext}_A^i(DA, N)=0$
for all $i\geq 1$.

We note the following analogue of Lemma \ref{lem:totallyacyclic}.

\begin{lem}\label{lem:cotatallyacyclic}
Let $I^\bullet$ be a complex of injective modules. Then the
following statements are equivalent:
\begin{enumerate}
\item[(1)] the complex $I^\bullet$ is cototally acyclic;
\item[(2)] the complex $I^\bullet$ is acyclic and each coboundary $B^i(P^\bullet)$
lies in $(DA)^\perp$;
\item[(3)] the complex $\nu^{-} I^\bullet$ is totally acyclic.
\end{enumerate}
\end{lem}

\begin{proof}
The proof is  analogous to the one of Lemma
\ref{lem:totallyacyclic}. Just note that for $(1)\Leftrightarrow
(3)$, one uses that the natural chain map
$\nu\nu^{-}I^\bullet\stackrel{\sim}\longrightarrow I^\bullet$ is an
isomorphism.
\end{proof}

\begin{rem} \label{rem:D}
Note that $\nu^{-} I^\bullet=(DI^\bullet)^*$. By Lemma
\ref{lem:totallyacyclic} we infer that $I^\bullet$ is cototally
acyclic if and only if $DI^\bullet$ is totally acyclic.
Consequently, a module $M$ is Gorenstein-injective if and only if
$(DM)_A$ is Gorenstein-projective. Moreover, one has a duality
$D\colon A\mbox{-{\rm Ginj}}\stackrel{\sim}\longrightarrow A^{\rm
op}\mbox{-{\rm Gproj}}$.
\end{rem}

Let us remark that from Lemma \ref{lem:cotatallyacyclic}(2) it
follows that $A\mbox{-Ginj}\subseteq (DA)^\perp$. Note that in the
case of the lemma above, we have that $Z^0(\nu^{-}I^\bullet)\simeq
\nu^{-}B^0(I^\bullet)$. Dually a complex $P^\bullet$ of projective
modules is totally acyclic if and only if the complex
$\nu P^\bullet$ is cototally acyclic; in this case, we have
$B^0(\nu P^\bullet)\simeq \nu Z^0(P^\bullet)$.

 We conclude from the above discussion the following result.

\begin{prop}\label{prop:GinjGproj}
Let $M$ be an $A$-module. Then we have:
\begin{enumerate}
\item[(1)] $M$ is Gorenstein-injective if and only if $\nu^{-} M$ is
Gorenstein-projective and the natural morphism $\nu\nu^{-}
M\rightarrow M$ is an isomorphism;
\item[(2)]  $M$ is Gorenstein-projective if and only if $\nu M$ is
Gorenstein-injective and the natural morphism $M\rightarrow
\nu^{-}\nu M$ is an isomorphism.
\end{enumerate}
Consequently, we have an equivalence $\nu^{-}\colon A\mbox{-{\rm
Ginj}}\stackrel{\sim}\longrightarrow A\mbox{-{\rm Gproj}}$ of
categories with its quasi-inverse given by $\nu$. \hfill $\square$
\end{prop}

The following result is analogous to Lemma \ref{lem:Gproj}.

\begin{lem}\label{lem:Ginj}
Let $N$ be an $A$-module. Then the following statements are
equivalent:
 \begin{enumerate}
 \item[(1)] $N$ is Gorenstein-injective;
 \item[(2)] there exists a long exact sequence $\cdots \rightarrow
 I^{-3} \rightarrow I^{-2} \rightarrow I^{-1}\rightarrow N\rightarrow 0$ with each $I^{-i}$ injective
  and each  coboundary  in $(DA)^\perp$;
 \item[(3)] $N\in {(DA)^\perp}$, ${\rm Tor}_i^A(DA, \nu^{-} N)=0$ for $i\geq 1$ and the natural
 morphism $\nu\nu^{-}N\rightarrow N$ is an isomorphism.
 \end{enumerate}
\end{lem}

\begin{proof} The proof is analogous to the one of  Lemma \ref{lem:Gproj}.
We apply Lemma \ref{lem:cotatallyacyclic} and note that we have
$Z^{0}(\nu^{-}I^\bullet)\simeq \nu^{-}B^0(I^\bullet)$ for a
cototally acyclic complex $I^\bullet$.
\end{proof}

 Recall that
$A\mbox{-inj}\subseteq A\mbox{-Ginj}$. Dual to Propostion
\ref{prop:Gproj} one can show that $A\mbox{-Ginj}\subseteq
A\mbox{-mod}$ is a \emph{coresolving subcategory}\index{subcategory!coresolving},
 that is, it contains all the injective modules and is closed under extensions,
 taking cokernels of
monomorphisms and direct summands. Since $A\mbox{-Ginj}$ is closed
under extensions, it becomes an exact category; moreover, it is
Frobenius such that its projective objects are equal to injective
$A$-modules.

Denote by $A\mbox{-}\overline{\mbox{mod}}$ the \emph{stable
category}\index{category!stable} of $A\mbox{-mod}$ modulo injective
modules. Recall that for a module $M$ the \emph{cosyzygy
module}\index{module!cosyzygy} $\Omega^{-}M$ is defined to be the
cokernel of a monomorphism $M\rightarrow I$ with $I$ injective; this
gives rise to the \emph{cosyzygy functor} \index{functor!cosyzygy}
$\Omega^{-}\colon A\mbox{-}\overline{\mbox{mod}}\rightarrow
A\mbox{-}\overline{\mbox{mod}}$. For each $i\geq 1$ denote by
$\Omega^{-i}$ the $i$-th power of $\Omega^{-}$.

 We denote by
$A\mbox{-}\overline{\mbox{Ginj}}$ the full subcategory of
$A\mbox{-}\overline{\mbox{mod}}$ consisting of Gorenstein-injective
modules. Dual to Corollary \ref{cor:syzygy} one observes that for a
Gorenstein-injective module $M$ its cozysygy modules $\Omega^{-i} M$
are Gorenstein-injective. In particular, one has an induced functor
$\Omega^{-}\colon A\mbox{-}\overline{\rm Ginj} \rightarrow
A\mbox{-}\overline{\rm Ginj}$.

\begin{lem}
The stable category $A\mbox{-}\overline{\mbox{{\rm Ginj}}}$ is
triangulated with the shift functor given by $\Omega^{-}$. Moreover,
we have an equivalence $\nu^{-}\colon A\mbox{-}\overline{\mbox{{\rm
Ginj}}}\stackrel{\sim}\longrightarrow A\mbox{-\underline{\rm
Gproj}}$ of triangulated categories with its quasi-inverse given by
$\nu$.
\end{lem}

\begin{proof}
The first statement is dual to Proposition
\ref{prop:Gprojtriangulated}(2), while the second follows from
Proposition \ref{prop:GinjGproj}. Note that the second statement
also follows from Corollary \ref{cor:trianglequivalence} and Remark
\ref{rem:D}.
\end{proof}

\begin{rem}
The cosyzygy functor  $\Sigma=\Omega^{-}$ is invertible on
$A\mbox{-}\overline{\mbox{{\rm Ginj}}}$; its quasi-inverse is given
by $\Sigma^{-1}=\nu\Omega\nu^{-}$.
\end{rem}

Recall that $\tau=D\circ {\rm Tr} \colon A\mbox{-\underline{mod}}
\stackrel{\sim}\longrightarrow A\mbox{-}\overline{\rm mod}$ is the
\emph{Auslander-Reiten translation};\index{Auslander-Reiten translation} it is an equivalence with its
quasi-inverse given by $\tau^{-1}={\rm Tr}\circ D$; see
\cite[p.106]{ARS}.

 The following result is of independent interest.

\begin{prop}\label{prop:ARtranslation}
Let $M$ be an $A$-module. Then $M$ is Gorenstein-projective if and
only if $\tau M$ is Gorenstein-injective; dually $M$ is
Gorenstein-injective if and only if $\tau^{-1} M$ is
Gorenstein-projective.
\end{prop}

\begin{proof}
We apply Proposition \ref{prop:transpose} and Remark \ref{rem:D}.
\end{proof}

From now on we will study for an artin algebra $A$ the category
$A\mbox{-Mod}$ of ``large" $A$-modules, that is, modules which are
not necessarily finitely generated. We will consider modules and
complexes in $A\mbox{-Mod}$.

Denote by $A\mbox{-Proj}$ (resp. $A\mbox{-Inj}$) the full
subcategory of $A\mbox{-Mod}$ consisting of projective (resp.
injective) modules. Note that the Nakayama functor induces an
equivalence $\nu\colon A\mbox{-Proj}\stackrel{\sim}\longrightarrow
A\mbox{-Inj}$ with its quasi-inverse given by $\nu^{-}$; see
\cite[Lemma 5.4]{Kr97}.

We note the following well-known fact.

\begin{lem}\label{lem:well-known}
Let $A$ be an artin algebra. Then a module is projective if and only
if it is a direct summand of a product of $_AA$; a module is
injective if  and only if it is a direct summand of a coproduct of
$D(A_A)$. \hfill $\square$
\end{lem}

An acyclic complex $P^\bullet$ of projective $A$-modules is
\emph{totally acyclic}\index{complex!totally acyclic} if for each
projective $A$-module $Q$ the Hom complex ${\rm Hom}_A(P^\bullet,
Q)$ is acyclic; dually an acyclic complex $I^\bullet$ of injective
$A$-modules is \emph{cototally acyclic}\index{complex!cototally
acyclic} if for each injective $A$-module $J$ the Hom complex ${\rm
Hom}_A(J, I^\bullet)$ is acyclic.

Denote by $^\perp(A\mbox{-{\rm Proj}})$ the full subcategory of
$A\mbox{-Mod}$ consisting of modules $M$ such that ${\rm Ext}^i_A(M,
Q)=0$ for all $i\geq 1$ and $Q$ projective.

We note the following equalities
\begin{align*}
^\perp(A\mbox{-{\rm Proj}})&=\{M\in
A\mbox{-Mod}\; |\; {\rm Ext}^i_A(M, A)=0 \mbox{ for all } i\geq
1\}\\
& = \{M\in A\mbox{-Mod}\; |\; {\rm Tor}_i^A(D(_AA), M)=0 \mbox{ for
all } i\geq 1\},
\end{align*}
where the first equality follows from Lemma \ref{lem:well-known} and
the second from the fact that $D{\rm Tor}_i^A(N, M)\simeq {\rm
Ext}_A^i(M, DN)$. It then follows that $^\perp(A\mbox{-{\rm
Proj}})\subseteq A\mbox{-Mod}$ is closed under products and
$^\perp(A\mbox{-{\rm Proj}})\cap A\mbox{-mod}={^\perp A}$.

The following is analogous to Lemma \ref{lem:cotatallyacyclic}.

\begin{lem}\label{lem:totallyacyclic,large} Let $P^\bullet$ be a complex of projective $A$-modules.
Then the following are equivalent:
\begin{enumerate}
\item[(1)] the complex $P^\bullet$ is totally acyclic;
\item[(2)] the complex $P^\bullet$ is acyclic and each cocycle $Z^i(P^\bullet)$
lies in $^\perp(A\mbox{-{\rm Proj}})$;
\item[(3)] the complexes $P^\bullet$ and $\nu P^\bullet$  are both
acyclic;
\item[(4)] the complex $\nu P^\bullet$ is cototally acyclic. \hfill
$\square$
\end{enumerate}
\end{lem}

The following notion was first introduced by Enochs and Jenda
(\cite{EJ1}).

\begin{defn}
A module $M$ is said to be  (large)
\emph{Gorenstein-projective}\index{Gorenstein-projective
module!large} provided that it is the zeroth cocycle of a totally
acyclic complex; a module $N$ is said to be (large)
\emph{Gorenstein-injective}\index{Gorenstein-injective module!large}
provided that it is the zero coboundary of a cototally acyclic
complex.
\end{defn}

 Denote
by $A\mbox{-GProj}$ (resp. $A\mbox{-GInj}$) the full subcategory of
$A\mbox{-Mod}$ consisting of Gorenstein-projective (resp.
Gorenstein-injective) modules. Note that $A\mbox{-Proj}\subseteq
A\mbox{-GProj}$ and $A\mbox{-Inj}\subseteq A\mbox{-GInj}$.

The following is analogous to Lemma \ref{lem:Gproj}.

\begin{lem} \label{lem:GProj} Let $M\in A\mbox{-{\rm Mod}}$. Then the following
statements are equivalent:
\begin{enumerate}
\item[(1)] $M$ is Gorenstein-projective;
\item[(2)] there is a long exact sequence $0\rightarrow M
\rightarrow P^0\rightarrow P^1 \rightarrow P^2 \rightarrow \cdots$
with each $P^i$ projective and each cocycle  in $^\perp(A\mbox{-{\rm
Proj}})$;
\item[(3)] ${\rm Tor}_i^A(DA, M)=0={\rm Ext}^i_A(DA, \nu M)$ for $i\geq
1$ and the natural morphism $M\rightarrow \nu^{-}\nu M$ is an
isomorphism. \hfill $\square$
\end{enumerate}
\end{lem}

Let us remark that we have an analogue of Proposition
\ref{prop:Gproj} for $A\mbox{-GProj}$. In particular,
$A\mbox{-GProj}$ is closed under taking direct summands.

The second part of the following result is contained in
\cite[Proposition 1.4]{Zh} (also see \cite[Lemma 3.4]{Ch2}).

\begin{prop}\label{prop:GProj}
Let $A$ be an artin algebra. Then we have
\begin{enumerate}
\item[(1)] the subcategory $A\mbox{-{\rm GProj}}\subseteq A\mbox{-{\rm Mod}}$ is
closed under coproducts, products and filtered colimits;
\item[(2)] $A\mbox{-{\rm GProj}}\cap A\mbox{-{\rm mod}}=A\mbox{-{\rm Gproj}}$.
\end{enumerate}
\end{prop}

\begin{proof}
 Note that the functors $\nu$, $\nu^{-}$, ${\rm Tor}_i^A(DA, -)$
and ${\rm Ext}_A^i(DA, -)$ commute with coproducts, products and
filtered colimits.  Then (1) follows from Lemma \ref{lem:GProj}(3).

For (2), note first that $^\perp A \subseteq {^\perp(A\mbox{-{\rm
Proj}})}$. It follows from Lemmas \ref{lem:Gproj}(2) and
\ref{lem:GProj}(2) that $A\mbox{-{\rm Gproj}}\subseteq A\mbox{-{\rm
GProj}}\cap A\mbox{-{\rm mod}}$. On the other hand, let $M\in
A\mbox{-{\rm GProj}}\cap A\mbox{-{\rm mod}}$. Take a short exact
sequence $0\rightarrow M \rightarrow P\rightarrow M'\rightarrow 0$
with $P$ projective and $M'$ Gorenstein-projective. We may assume
that $P$ is free. Since $M$ is finitely generated, there is a
decomposition $P\simeq P^0\oplus P'$ such that $P^0$ is finitely
generated containing $M$. Then we have a short exact sequence
$0\rightarrow M\rightarrow P^0\rightarrow M^1\rightarrow 0$. Note
that $M^1\oplus P'\simeq M'$ and that $M'$ is Gorenstein-projective.
Then $M^1$ is also Gorenstein-projective. Therefore $M^1 \in
A\mbox{-{\rm GProj}}\cap A\mbox{-{\rm mod}}$. Observe that $M^1\in
{^\perp A}$. Repeat the argument with $M$ replaced by $M^1$. We get
a long exact sequence $0\rightarrow M\rightarrow P^0\rightarrow
P^1\rightarrow \cdots$. Now we apply Lemma \ref{lem:Gproj}(2).
\end{proof}

Note that we have a version of Proposition \ref{prop:GinjGproj} for
large modules. In particular, there is an equivalence $\nu\colon
A\mbox{-GProj}\stackrel{\sim}\longrightarrow A\mbox{-GInj}$ with its
quasi-inverse given by $\nu^{-}$.

 Note that Auslander-Reiten
translations allow a natural extension on $A\mbox{-Mod}$ as follows:
for an $A$-module $M$ take a projective presentation
$P^{-1}\rightarrow P^0\rightarrow M\rightarrow 0$ and define $\tau
M$ to be the kernel of the induced morphism $\nu P^{-1}\rightarrow
\nu P^0$. Similarly one extends $\tau^{-1}$. Then we have a duality
$\tau\colon A\mbox{-\underline{Mod}}\stackrel{\sim}\longrightarrow
A\mbox{-}\overline{\rm Mod}$ with its quasi-inverse given by
$\tau^{-1}$\index{Auslander-Reiten translation}; for details, see
\cite[section 5]{Kr97} and \cite[Remark 2.3(ii)]{Bel2}.

The following result extends Proposition \ref{prop:ARtranslation};
see \cite[Proposition 3.4]{Bel2}.

\begin{prop}
Let $M$ be an $A$-module. Then $M$ is Gorenstein-projective if and
only if $\tau M$ is Gorenstein-injective; $M$ is
Gorenstein-injective if and only if $\tau^{-1} M$ is
Gorenstein-projective.
\end{prop}

\begin{proof}
The proof is similar as the one of Proposition \ref{prop:transpose}.
The following is analogous to \cite[Chapter IV, Proposition
3.2]{ARS}: for an $A$-module $M$, we have the following exact
sequence
$$0\longrightarrow {\rm Ext}_A^1(DA, \tau M)\longrightarrow M\longrightarrow \nu^{-}\nu M\longrightarrow
{\rm Ext}^2_A(DA, \tau M)\longrightarrow 0,$$ where the middle
morphism is the natural map associated to the adjoint pair $(\nu,
\nu^{-})$. Then we apply  a version  of  Propositions
\ref{prop:Gproj}(3) and
 \ref{prop:GinjGproj} for large modules.
\end{proof}

We make the following observation; see \cite[Lemma 8.6]{Bel2}.

\begin{prop}
Let $M$ (resp. $N$) be a Gorenstein-projective (resp.
Gorenstein-injective) $A$-module. Then we have
\begin{enumerate}
\item[(1)] $DM$ (resp. $M^*$) is a right Gorenstein-injective  (resp.
Gorenstein-projective) $A$-module;
\item[(2)]  $DN$ is a right Gorenstein-projective $A$-module.
\end{enumerate}
\end{prop}

\begin{proof}
Observe that for a complex $P^\bullet$ of projective modules, we
have $\nu^{-} DP^\bullet\simeq D\nu P^\bullet$. It follows that for
a totally acyclic complex $P^\bullet$ the complex $DP^\bullet$ is
cototally acyclic; see Lemma \ref{lem:cotatallyacyclic}. Then for a
Gorenstein-projective $A$-module $M$, $DM$ is Gorenstein-injective.
Dually observe that for a cototally acyclic complex $I^\bullet$
there is a totally acyclic complex $P^\bullet$ such that  $\nu
P^\bullet=I^\bullet$, and note that $DI^\bullet=\nu^{-}DP^\bullet$
is totally acyclic; see Lemma \ref{lem:cotatallyacyclic}. This
proves that for a Gorenstein-injective module $N$, $DN$ is
Gorenstein-projective. Note that $M^*=D\nu M$. Since $\nu M$ is
Gorenstein-injective, then by (2) $M^*$ is Gorenstein-projective.
\end{proof}

The category $A\mbox{-GProj}$ is a Frobenius exact category such
that its projective objects are equal to (large) projective
$A$-modules. We denote by $A\mbox{-\underline{GProj}}$ its stable
category modulo projective modules which is triangulated; compare
Proposition \ref{prop:Gprojtriangulated}.

 Note that the shift functor on
$A\mbox{-\underline{GProj}}$ is given by
$\Sigma=\nu^{-}\Omega^{-}\nu$ whose quasi-inverse is given by
$\Omega$ (compare Remark \ref{rem:shiftfunctor}). Note that the inclusion
$A\mbox{-Gproj}\subseteq A\mbox{-GProj}$ induces a fully faithful
triangle functor $A\mbox{-\underline{\rm Gproj}} \hookrightarrow
A\mbox{-\underline{\rm GProj}}$.

Dually we have that the stable category $A\mbox{-}\overline{\rm
GInj}$ of $A\mbox{-GInj}$ modulo injective $A$-modules is
triangulated and that the inclusion $A\mbox{-inj}\hookrightarrow
A\mbox{-GInj}$ induces a full embedding $A\mbox{-}\overline{\rm
Ginj} \hookrightarrow A\mbox{-}\overline{\rm GInj}$ of triangulated
categories.

Observe that the equivalence $\nu\colon
A\mbox{-GProj}\stackrel{\sim}\longrightarrow A\mbox{-GInj}$ and its
quasi-inverse $\nu^{-}$ are both exact. Hence we have

\begin{lem}
Let $A$ be an artin algebra. Then the Nakayama functor induces a
triangle equivalence $\nu\colon A\mbox{-\underline{\rm
GProj}}\stackrel{\sim}\longrightarrow A\mbox{-}\overline{\rm GInj}$
with its quasi-inverse given by $\nu^{-}$. \hfill $\square$
\end{lem}

Recall that in a triangulated category $\mathcal{T}$ with arbitrary
coproducts an object $C$ is \emph{compact}\index{compact object} if
the functor ${\rm Hom}_\mathcal{T}(C, -)\colon
\mathcal{T}\rightarrow {\rm Ab}$ commutes with coproducts. Here
``Ab" denotes the category of abelian groups. Denote by
$\mathcal{T}^c$ the full subcategory consisting of compact objects;
it is a thick triangulated subcategory.

The following observation is easy; see Proposition
\ref{prop:GProj}(1).

\begin{lem}
Let $A$ be artin algebra. Then the triangulated category
$A\mbox{-\underline{\rm GProj}}$ has arbitrary coproducts and
products, and the natural full embedding $A\mbox{-\underline{\rm
Gproj}} \hookrightarrow A\mbox{-\underline{\rm GProj}}$ induces
$A\mbox{-\underline{\rm Gproj}} \hookrightarrow
(A\mbox{-\underline{\rm GProj}})^c$. \hfill $\square$
\end{lem}

In what follows we will discuss very briefly strongly
Gorenstein-projective modules. They play the role as ``free objects"
in Gorenstein homological algebra. Let us first study finitely
generated strongly Gorenstein-projective modules.

Let $A$ be an artin algebra and let $n\geq 1$. Following Bennis and
Mahdou (\cite{BM2}) a totally acyclic complex is said to  be
\emph{$n$-strong} \index{complex!totally acyclic!$n$-strong}if it is
of the following form
$$\cdots \rightarrow P^{-1}\stackrel{d^{n-1}}\rightarrow P^0
\stackrel{d^0}\rightarrow P^1\stackrel{d^1}\rightarrow P^2
\rightarrow \cdots \rightarrow P^{n-1}\stackrel{d^{n-1}}\rightarrow
P^n \stackrel{d^0}\rightarrow P^{n+1}\rightarrow \cdots$$

A finitely generated $A$-module $M$ is said to be \emph{$n$-strongly
Gorenstein-projective}\index{Gorenstein-projective
module!$n$-strongly} provided that it is the zeroth cocycle of a
totally acyclic complex which is $n$-strong. Denote by
$n\mbox{-}A\mbox{-SGproj}$ the full subcategory of $A\mbox{-mod}$
consisting of such modules. $1$-strongly Gorenstein projective
modules are called \emph{strongly
Gorenstein-projective}\index{Gorenstein-projective module!strongly}
and $1\mbox{-}A\mbox{-SGproj}$ is also denoted by $A\mbox{-SGproj}$
(\cite{BM}). Observe that if $n$ divides $m$ we have
$n\mbox{-}A\mbox{-SGproj}\subseteq m\mbox{-}A\mbox{-SGproj}$.
Observe that a projective module $P$ is strongly
Gorenstein-projective, since we may take its complete
resolution $$\cdots \rightarrow P\oplus P \xrightarrow{\begin{pmatrix}0 & 1 \\
0 & 0
\end{pmatrix}} P\oplus P \xrightarrow{\begin{pmatrix}0 & 1 \\
0 & 0
\end{pmatrix}} P\oplus P \rightarrow \cdots$$
which is $1$-strong. Then we have  $A\mbox{-proj}\subseteq
A\mbox{-SGproj} \subseteq n\mbox{-}A\mbox{-SGproj}\subseteq
A\mbox{-Gproj}$.

The following characterizes $n$-strongly Gorenstein-projective
modules.

\begin{prop}\label{prop:strongGproj}
An $A$-module $M$ is $n$-strongly Gorenstein-projective if and only
if $\Omega^n M\simeq M$ (in the stable category) and  ${\rm
Ext}^i_A(M, A)=0$ for $1\leq i\leq n$.
\end{prop}

\begin{proof}
The ``only if" part is easy. For the ``if" part, by dimension-shift
we infer that $M\in {^\perp A}$. Take an exact sequence
$0\rightarrow K\rightarrow P^{n-1}\rightarrow \cdots \rightarrow
P^0\rightarrow M\rightarrow 0$ such that each $P^i$ is projective.
By assumption $K$ and $M$ are isomorphic in
$A\mbox{-\underline{mod}}$. Then there exist projective modules $P$
and $Q$ such that $K\oplus P\simeq M\oplus Q$. By adding $P$ to
$P^{n-1}$ we may assume that $K\simeq M\oplus Q$.  Denote by $M'$
the image of $P^{n-1}\rightarrow P^{n-2}$. By dimension-shift we
have ${\rm Ext}^1_A(M', Q)\simeq {\rm Ext}^{n}_A(M, Q)=0$. Consider
the short exact sequence $0\rightarrow M\oplus Q \rightarrow
P^{n-1}\rightarrow M'\rightarrow 0$. We conclude from ${\rm
Ext}_A^1(M', Q)=0$ that there is a decomposition $P^{n-1}=P'\oplus
Q$ such that there is a  short exact sequence $0\rightarrow
M\rightarrow P'\rightarrow M'\rightarrow 0$. Then we get a long
exact sequence $0\rightarrow M\rightarrow P'\rightarrow
P^{n-2}\rightarrow \cdots \rightarrow P^0\rightarrow M\rightarrow 0$
from which we construct an $n$-strong complete resolution for $M$
immediately.
\end{proof}

We note the following immediate consequence.

\begin{cor}\label{cor:strongGproj}
An $A$-module $M$ is a direct summand of an $n$-strongly
Gorenstein-projective module if and only if $M$ is
Gorenstein-projective and $\Omega^{nd} M\simeq M$ for some $d\geq
1$.
\end{cor}

\begin{proof}
For the ``if" part, take $N=\oplus_{i=0}^{d-1} \Omega^{ni} M$; it is
$n$-strongly Gorenstein-projective by Proposition
\ref{prop:strongGproj}. Conversely, assume that $M$ is a direct
summand of an $n$-strongly Gorenstein-projective $N$. It suffices to
show the result in the case that $M$ is indecomposable. Note that by
Proposition \ref{prop:strongGproj} we have $\Omega^{ni}N\simeq N$
and then for all $i\geq 1$, $\Omega^{ni}M$ are direct summands of
$N$. By Krull-Schmidt Theorem we infer that $M\simeq \Omega^{nd} M$
for some $d\geq 1$.
\end{proof}

For an additive subcategory $\mathcal{X}$ of $A\mbox{-mod}$, denote
by ${\rm add}\; \mathcal{X}$ its \emph{additive closure}\index{additive closure}, that is,
${\rm add}\; \mathcal{X}$ consists of direct summands of the modules
in $\mathcal{X}$. Then it follows from the above results that
$${\rm add}\;
A\mbox{-SGproj}=\bigcup_{n\geq 1} \; n\mbox{-}A\mbox{-SGproj}.$$

There is a large version of $n$-strongly Gorenstein-projective
$A$-modules. The subcategory of $A\mbox{-Mod}$ consisting of (large)
$n$-strongly Gorenstein-projective modules is denoted by
$n\mbox{-}A\mbox{-SGProj}$; $1\mbox{-}A\mbox{-SGProj}$ is also
denoted by $A\mbox{-SGProj}$. As above we have inclusions
$A\mbox{-Proj}\subseteq A\mbox{-SGProj} \subseteq
n\mbox{-}A\mbox{-SGProj}\subseteq A\mbox{-GProj}$.

Note that Proposition \ref{prop:strongGproj} works for any module
$M\in A\mbox{-Mod}$. Then it follows that
$n\mbox{-}A\mbox{-SGProj}\cap A\mbox{-mod}=n\mbox{-}A\mbox{-SGproj}$
for all $n\geq 1$.

The following result is of interest; see \cite[Theorem 2.7]{BM}.
Note that  Corollary \ref{cor:strongGproj} does not apply to large
modules.

\begin{prop} \label{prop:SGdirectsummands}
Let $A$ be an artin algebra. Then we have ${\rm add}\; A\mbox{-{\rm
SGProj}}=A\mbox{-{\rm GProj}}$.
\end{prop}

\begin{proof}
The inclusion ${\rm add}\; A\mbox{-{\rm SGProj}}\subseteq
A\mbox{-{\rm GProj}}$ is clear. On the other hand, we need to show
that each Gorenstein-projective module $M$ is a direct summand of a strong Gorenstein-projective
module. Take a complete resolution $P^\bullet$ for $M$.  For each $i\in \mathbb{Z}$
denote by $P^\bullet(i)$ the shifted complex of $P^\bullet$, which is defined by
$(P^\bullet(i))^n=P^{n+i}$ and $d^n_{P(i)}=d_P^{n+i}$.
Consider the
complex $\oplus_{i\in \mathbb{Z}} P^\bullet(i)$; it is a strong totally acyclic
complex. Note that its zeroth cocycle is $N= \oplus_{i\in \mathbb{Z}} Z^i(P^\bullet)$ and by
definition $N$ is strong Gorenstein-projective. Observe that $M$ is a direct summand of $N$.
\end{proof}

\section{Gorenstein Algebras}

In this section we will study Gorenstein-projective modules over
Gorenstein algebras. This is the case where  Gorenstein-projective
modules behave the best. Other related notions such as virtually
Gorenstein algebras, CM-finite algebras and CM-free algebras will be
discussed very briefly.

Recall that an artin algebra $A$ is \emph{self-injective}\index{algebra!self-injective} provided
that its regular module $_AA$ is injective; this is equivalent to
that projective modules are injective and vice verse; see
\cite[Chapter IV, section 3]{ARS}.

The following result is easy.

\begin{prop}\label{prop:self-injective}
Let $A$ be an artin algebra. Then the following statements are
equivalent:
\begin{enumerate}
\item[(1)] the algebra $A$ is self-injective;
\item[(2)] $A\mbox{-{\rm mod}}=A\mbox{-{\rm Gproj}}$;
\item[(2')] $A\mbox{-{\rm mod}}=A\mbox{-{\rm Ginj}}$;
\item[(3)] $A\mbox{-{\rm inj}}\subseteq A\mbox{-{\rm Gproj}}$;
\item[(3')] $A\mbox{-{\rm proj}}\subseteq A\mbox{-{\rm Ginj}}$.
\end{enumerate}
\end{prop}

\begin{proof} Note that for a self-injective algebra $A$ and an $A$-module $M$,
we may splice the projective resolution and the injective resolution
of $M$ to get a complete resolution for $M$. This shows ``
$(1)\Rightarrow(2)$". The implication ``$(2)\Rightarrow (3)$" is
trivial. For ``$(3)\Rightarrow (1)$", note that then $D(A_A)$ is
Gorenstein-projective, in particular, it is a submodule of a
projective module. Since $D(A_A)$ is injective, the submodule is
necessarily a direct summand. Hence $D(A_A)$ is projective and then
the algebra $A$ is self-injective.
\end{proof}

Recall that an artin algebra $A$ is
\emph{Gorenstein}\index{algebra!Gorenstein} provided that the
regular module $A$ has finite injective dimension on both sides
(\cite{Hap2}).  We have that an algebra $A$ is Gorenstein is
equivalent to that any $A$-module has finite projective dimension if
and only if it has finite injective dimension.

Observe that for a Gorenstein algebra $A$ we have ${\rm inj.dim}\;
_AA={\rm inj.dim}\; A_A$ (\cite[Lemma 6.9]{AR}); the common value is
denoted by ${\rm G.dim}\; A$. If ${\rm G.dim}\; A\leq d$, we say
that $A$ is \emph{$d$-Gorenstein}\index{algebra!Gorenstein!$d$-Gorenstein}.\footnote{In the literature there
is a different notion of $d$-Gorenstein algebra; see \cite{AR1}.}
Note that $0$-Gorenstein algebras are the same as self-injective
algebras. An algebra of finite global dimension $d$ is
$d$-Gorenstein.

Let us begin with the following observation.

\begin{lem}\label{lem:fpd=fid}
Let $A$ be a $d$-Gorenstein algebra and let $M\in A\mbox{-{\rm
Mod}}$. If $M$ has finite projective dimension, then ${\rm
proj.dim}\; M\leq d$ and ${\rm inj.dim}\; M\leq d$.
\end{lem}

\begin{proof}
Note that ${\rm inj.dim} \; _AA={\rm proj.dim}\; D(A_A)\leq d$. We
use the following fact: an $A$-module $M$ of finite projective
dimension $n$ satisfies that ${\rm Ext}_A^n(M,Q)\neq 0$ for some
projective $A$-module $Q$; by Lemma \ref{lem:well-known} this is
equivalent to ${\rm Ext}_A^n(M, A)\neq 0$. This shows that ${\rm
proj.dim}\; M\leq d$. Similarly one shows that ${\rm inj.dim}\;
M\leq d$.
\end{proof}

 For each $d\geq 1$ denote by $\Omega^d(A\mbox{-mod})$ the class of
modules of the form $\Omega^d M$ for a module $M$. In addition we
identify $\Omega^0(A\mbox{-mod})$ with $A\mbox{-mod}$. Note that
$A\mbox{-Gproj}\subseteq \Omega^d(A\mbox{-mod})$ for all $d\geq 0$.
Dually we have the notations  $\Omega^{-d}(A\mbox{-mod})$ for $d\geq
0$.

The following result, which is implicitly contained in \cite[Theorem
3.2]{AM}, characterizes $d$-Gorenstein algebras; it generalizes part
of  of Proposition \ref{prop:self-injective}.

\begin{thm}\label{thm:GorensteinalgebraI}
Let $A$ be an artin algebra and let $d\geq 0$. Then the following
statements are equivalent:
\begin{enumerate}
\item[(1)] the algebra $A$ is $d$-Gorenstein;
\item[(2)] $A\mbox{-{\rm Gproj}}=\Omega^d(A\mbox{-{\rm mod}})$;
\item[(2')] $A\mbox{-{\rm Ginj}}=\Omega^{-d}(A\mbox{-{\rm mod}})$.
\end{enumerate}
In this case, we have $A\mbox{-{\rm Gproj}}={^\perp A}$ and
$A\mbox{-{\rm Ginj}}=(DA)^\perp$.
\end{thm}

\begin{proof}
We only show the result concerning Gorenstein-projective modules.

For ``$(1)\Rightarrow (2)$", assume that $A$ is $d$-Gorenstein. Note
that by dimension-shift we have $\Omega^{d}(A\mbox{-mod})\subseteq
{^\perp A}$. We have already observed that $A\mbox{-Gproj}\subseteq
\Omega^d(A\mbox{-mod})$. Hence it suffices to show that any module
$M$ in $^\perp A$ is Gorenstein-projective. For this, take a
projective resolution $\cdots \rightarrow P^{-1}\rightarrow
P^0\rightarrow M\rightarrow 0$. By assumption we apply $(-)^*$ to
get an exact sequence $\xi\colon 0\rightarrow M^*\rightarrow
(P^0)^*\rightarrow (P^{-1})^*\rightarrow \cdots$. Since $A_A$ has
finite injective dimension, using dimension-shift on $\xi$ we infer
that $M^*\in {^\perp (A_A)}$; moreover, all the cocylces in $\xi$
lie in $^\perp (A_A)$. Hence applying $(-)^*$ to $\xi$ we still get
an exact sequence. Note that each $P^{-i}$ is reflexive. From this
we conclude that $M$ is reflexive. Then by Lemma \ref{lem:Gproj}(3)
$M$ is Gorenstein-projective.

To show ``$(2)\Rightarrow (1)$", assume that $A\mbox{-{\rm
Gproj}}=\Omega^d(A\mbox{-{\rm mod}})$. For each module $M$ and
$k\geq 1$, we have ${\rm Ext}_A^{d+k}(M, A)\simeq {\rm
Ext}_A^k(\Omega^d M, A)=0$, since $\Omega^d M$ is
Gorenstein-projective. Therefore ${\rm inj.dim}\; _AA\leq d$. On the
other hand, consider the long exact sequence $0\rightarrow \Omega^d
D(A_A)\rightarrow P^{1-d}\rightarrow \cdots \rightarrow
P^{-1}\rightarrow P^0\rightarrow D(A_A)\rightarrow 0$. For any
Gorenstein-projective module $G$, by dimension-shift we infer that
${\rm Ext}^{d+1}_A(G, \Omega^d D(A_A))\simeq {\rm Ext}^1_A(G,
D(A_A))=0$ (we can apply dimension-shift because of $G\in {^\perp
A}$). Note that by assumption $\Omega^d D(A_A)$ is
Gorenstein-projective, and then we may take a long exact sequence
$0\rightarrow \Omega^d D(A_A)\stackrel{\varepsilon}\rightarrow
Q^0\rightarrow Q^1\rightarrow \cdots \rightarrow Q^d \rightarrow
G\rightarrow 0$ with each $Q^i$ projective and $G$
Gorenstein-projective. Take $G'$ to be the cokernel of
$\varepsilon$. Then by dimension-shift again we have ${\rm
Ext}_A^1(G', \Omega^d D(A_A))\simeq {\rm Ext}_A^{d+1}(G, \Omega^d
D(A_A))=0$. It follows that $\varepsilon$ splits and then $\Omega^d
D(A_A)$ is projective. Hence ${\rm proj.dim}\; D(A_A)\leq d$ and
then ${\rm inj.dim}\; A_A\leq d$, completing the proof.

The equality $A\mbox{-{\rm Gproj}}={^\perp A}$ is shown in the proof
of ``$(1)\Rightarrow (2)$".
\end{proof}

\begin{rem}
For a Gorenstein algebra $A$, the modules in $^\perp A$ are often
called \emph{maximal Cohen-Macaulay modules}
(\cite{Buc})\index{module!maximal Cohen-Macaualy}.
\end{rem}

For an artin algebra $A$ and $n\geq 0$, denote by $\mathcal{P}^{\leq
n}(A\mbox{-mod})$ the full subcategory of $A\mbox{-mod}$ consisting
of modules having projective dimension at most $n$. Denote by
$\mathcal{P}^{<\infty}(A\mbox{-mod})$ the union of these categories.
Dually we have the notations  $\mathcal{I}^{\leq n}(A\mbox{-mod})$
and $\mathcal{I}^{<\infty}(A\mbox{-mod})$.

For a $d$-Gorenstein algebra $A$, by Lemma \ref{lem:fpd=fid} we have
$$\mathcal{P}^{<\infty}(A\mbox{-mod})=\mathcal{P}^{\leq
d}(A\mbox{-mod})=\mathcal{I}^{\leq
d}(A\mbox{-mod})=\mathcal{I}^{<\infty}(A\mbox{-mod}).$$

By Theorem \ref{thm:GorensteinalgebraI} we may apply
Auslander-Buchweitz's result (Theorem \ref{thm:AB}) to obtain the
following important result; compare \cite[Proposition 3.10]{Bel2}.

\begin{prop}\label{prop:cotorsionpair}
Let $A$ be a $d$-Gorenstein algebra. Then $(A\mbox{-{\rm Gproj}},
\mathcal{P}^{\leq d}(A\mbox{-{\rm mod}}))$ and $(\mathcal{I}^{\leq
d}(A\mbox{-{\rm mod}}), A\mbox{-{\rm Ginj}})$ are cotorsion pairs in
$A\mbox{-{\rm  mod}}$. \hfill $\square$
\end{prop}

Recall that a full additive subcategory $\mathcal{X}\subseteq
A\mbox{-mod}$ is said to be \emph{contravariantly
finite}\index{subcategory!contravariantly finite} provided that each
module $M$ admits a morphism $f_M\colon X_M\rightarrow M$ with
$X_M\in \mathcal{X}$ such that each morphism from a module in
$\mathcal{X}$ to $M$ factors through $f_M$. Such a morphism $f_M$ is
called a \emph{right $\mathcal{X}$-approximation}
\index{approximation!right}of $M$. Dually one has the notion of
\emph{covariantly finite subcategory}\index{subcategory!covariantly
finite}. A subcategory is said to be \emph{functorially
finite}\index{subcategory!functorially finite} provided that it is
both contravariantly finite and covariantly finite.

The following is well known; compare \cite[Corollary 5.10(1)]{AR}.
It is contained in \cite[Theorem 5]{BK}.

\begin{cor}
Let $A$ be a $d$-Gorenstein artin algebra. Then all the three
subcategories $A\mbox{-{\rm Gproj}}$, $\mathcal{P}^{\leq
d}(A\mbox{-{\rm mod}})=\mathcal{I}^{\leq d}(A\mbox{-{\rm mod}})$,
$A\mbox{-{\rm Ginj}}$ are functorially finite in $A\mbox{-{\rm
mod}}$.
\end{cor}

\begin{proof}
It follows from the two cotorsion pairs above that $A\mbox{-Gproj}$
is contravriantly finite, $\mathcal{P}^{\leq d}(A\mbox{-{\rm
mod}})=\mathcal{I}^{\leq d}(A\mbox{-{\rm mod}})$ is functorially
finite and $A\mbox{-Ginj}$ is covariantly finite. The rest follows
from  the following fact (and its dual): a resolving contravariantly
finite subcategory in $A\mbox{-mod}$ is functorially finite; see
\cite[Corollary 0.3]{KS}.
\end{proof}

There are analogues of the results above for large modules. For each
$d\geq 0$ denote by $\Omega^d(A\mbox{-Mod})$ the class of modules of
the form $\Omega^d M$ for an $A$-module $M$. Similarly we have the
notation $\Omega^{-d}(A\mbox{-Mod})$.

Then we have the following result.

\begin{thm}\label{thm:GorensteinalgebraII}
Let $A$ be an artin algebra and let $d\geq 0$. Then the following
statements are equivalent:
\begin{enumerate}
\item[(1)] the algebra $A$ is $d$-Gorenstein;
\item[(2)] $A\mbox{-{\rm GProj}}=\Omega^d(A\mbox{-{\rm Mod}})$;
\item[(2')] $A\mbox{-{\rm GInj}}=\Omega^{-d}(A\mbox{-{\rm Mod}})$.
\end{enumerate}
In this case, we have $A\mbox{-{\rm GProj}}=\{M\in A\mbox{-{\rm
Mod}}\; |\; {\rm Ext}_A^i(M, A)=0, i\geq 1\}$ and $A\mbox{-{\rm
GInj}}=\{M\in A\mbox{-{\rm Mod}}\; |\; {\rm Ext}_A^i(DA, M)=0, i\geq
1\}$. Moreover, we have two cotorsion pairs $(A\mbox{-{\rm GProj}},
\mathcal{P}^{\leq d}(A\mbox{-{\rm Mod}}))$ and $(\mathcal{I}^{\leq
d}(A\mbox{-{\rm Mod}}), A\mbox{-{\rm GInj}})$  in $A\mbox{-{\rm
Mod}}$.
\end{thm}

\begin{proof}
We just comment on the proof of  the results concerning
Gorenstein-projective modules. Note that the condition (2) implies
that $A\mbox{-Gproj}=\Omega^d(A\mbox{-mod})$ and then
``$(2)\Rightarrow (1)$" follows from Theorem
\ref{thm:GorensteinalgebraI}. To see ``$(1)\Rightarrow (2)$", first
observe that $A\mbox{-GProj}\subseteq
\Omega^d(A\mbox{-Mod})\subseteq \{M\in A\mbox{-{\rm Mod}}\; |\; {\rm
Ext}_A^i(M, A)=0, i\geq 1\}$. Take an $A$-module $M$ with the
property ${\rm Ext}_A^i(M, A)=0$ for all $i\geq 1$. We have to show
that $M$ is Gorenstein-projective. We apply Lemma
\ref{lem:GProj}(3). Note that $M$ satisfies that ${\rm Tor}_i^A(DA,
M)=0$ for all $i\geq 1$. We replace $(-)^*$ by $\nu$ in the proof of
``$(1)\Rightarrow (2)$" in Theorem \ref{thm:GorensteinalgebraI}.
\end{proof}

Following Beligiannis and Reiten an artin algebra $A$ is said to be
\emph{virtually Gorenstein}\index{algebra!virtually Gorenstein} provided that $(A\mbox{-{\rm
GProj}})^\perp={^\perp(A\mbox{-GInj})}$; see \cite[Chapter X,
Definition 3.3]{BR}. Then it follows from the two cotorsion pairs
above that a Gorenstein artin algebra is virtually Gorenstein. For
more on virtually Gorenstein algebras, see \cite{Bel2, BK}.

For an additive subcategory $\mathcal{X}$ of $A\mbox{-Mod}$ we
denote by $\underrightarrow{\rm lim}\; \mathcal{X}$ the full
subcategory of $A\mbox{-Mod}$ consisting of filtered colimits of
modules in $\mathcal{X}$, or equivalently, consisting of direct
limits of modules in $\mathcal{X}$ (\cite[Chapter 1, Theorem
1.5]{AR}). By Proposition \ref{prop:GProj} the full subcategory
$A\mbox{-GProj}$ is closed under filtered colimits. In particular we
have $\underrightarrow{\rm lim}\; A\mbox{-Gproj}\subseteq
A\mbox{-GProj}$. Similarly we have $\underrightarrow{\rm lim}\;
A\mbox{-Ginj}\subseteq A\mbox{-GInj}$.

The following result is of interest. It is contained in
\cite[Theorem 5]{BK}, while its first part is contained in
\cite{EJ?}.

\begin{prop}
Let $A$ be a Gorenstein artin algebra. Then we have
$\underrightarrow{\rm lim}\; A\mbox{-{\rm Gproj}}=A\mbox{-{\rm
GProj}}$ and $\underrightarrow{\rm lim}\; A\mbox{-{\rm
Ginj}}=A\mbox{-{\rm GInj}}$.
\end{prop}

\begin{proof}
Assume that the algebra $A$ is d-Gorenstein. Recall
from Proposition \ref{prop:cotorsionpair}
 the cotorsion pair $(A\mbox{-Gproj}, \mathcal{P}^{\leq d
}(A\mbox{-mod}))$ in $A\mbox{-mod}$. Then by \cite[Theorem
2.4(2)]{KS} we have $\underrightarrow{\rm lim}\; A\mbox{-{\rm
Gproj}}=\{M\in A\mbox{-Mod}\; |\; {\rm Ext}_A^i(M, L)=0, i\geq 1,
L\in \mathcal{P}^{\leq d}(A\mbox{-mod})\}$. By dimension-shift one
infer that an $A$-module $M$ has the property that ${\rm Ext}_A^i(M,
L)=0$ for all $ i\geq 1$ and all $ L$ in $\mathcal{P}^{\leq
d}(A\mbox{-mod})$ if and only if $M$ satisfies ${\rm Ext}_A^i(M,
A)=0$ for all $i\geq 1$. Then it follows from Theorem
\ref{thm:GorensteinalgebraII} that this is equivalent to that $M$
lies in $A\mbox{-GProj}$.
\end{proof}

Recall that finitely generated Gorenstein-projective modules over an
artin algebra $A$ are also called maximal Cohen-Macaulay modules.
Following \cite[Example 8.4(2)]{Bel2} an artin algebra $A$ is said
to be \emph{CM-finite}\index{algebra!CM-finite} provided that up to
isomorphism there are only finitely many indecomposable modules in
$A\mbox{-Gproj}$. Observe that an algebra of finite representation
type is CM-finite. By Corollary \ref{cor:duality} $A$ is CM-finite
if and only if so is $A^{\rm op}$.

Recently CM-finite Gorenstein artin algebras attract considerable
attentions; see \cite{Ch1, Bel3, LZ1}. It is believed that the
theory of Gorenstein-projective modules over a CM-finite Gorenstein
artin algebra is the simplest one among all the interesting cases.

As an extreme case of CM-finite algebras, we call that an artin
algebra $A$ is \emph{CM-free}\index{algebra!CM-free} provided that
$A\mbox{-Gproj}=A\mbox{-proj}$ (compare \cite{MO}). Roughly
speaking, for these kinds of algebras the theory of
Gorenstein-projective modules is boring. Note that by Corollary
\ref{cor:infinitedimension}(2) an algebra of finite global dimension
is CM-finite. On the other hand, by Theorem
\ref{thm:GorensteinalgebraI}(2) a Gorenstein algebra is CM-finite if
and only if it has finite global dimension.

Take $k$ to be a field and consider the three dimensional truncated
polynomial algebra $A=k[x,y]/(x^2, xy, y^2)$. It is well-known that
the algebra $A$ is CM-free. In what follows we will give a general
result, which generalizes \cite[Proposition 2.4]{Yoh}.

Recall that for an artin algebra $A$ its \emph{Ext-quiver}\index{Ext-quiver} $Q(A)$ is
defined such that the vertices are given by a complete set $\{S_1,
S_2, \cdots, S_n\}$ of pairwise non-isomorphic simple $A$-modules
and there is an arrow from $S_i$ to $S_j$ if and only if ${\rm
Ext}_A^1(S_i, S_j)\neq 0$. Recall that the algebra $A$ is connected
if and only if the underlying graph of $Q(A)$ is connected.

\begin{thm}
Let $A$ be a connected artin algebra such that $\mathbf{r}^2=0$.
Here $\mathbf{r}$ is the Jacobson radical of $A$. Then either $A$ is
self-injective or CM-free.
\end{thm}

\begin{proof}
Assume that $A$ is not CM-finite. Take $M\in A\mbox{-Gproj}$ to be
non-projective and indecomposable. Note that there is a short exact
sequence $0\rightarrow M\rightarrow P\stackrel{\pi}\rightarrow
M'\rightarrow 0$ with $P$ projective and $M'\in A\mbox{-Gproj}$. It
follows that $\pi$ is a projective cover and then $M\subseteq
\mathbf{r}P$. Note that $\mathbf{r}^2=0$. Hence $\mathbf{r}M=0$ and
then $M$ is semisimple. Note that $M$ is indecomposable. Then we
conclude that $M$ is a simple module.

Let $S_1=M$ be the above simple module. Take a short exact sequence
$\eta\colon 0\rightarrow S_2\stackrel{i_2}\rightarrow
P_1\stackrel{\pi_1}\rightarrow S_1\rightarrow 0$ such that $\pi_1$
is a projective cover. Observe that $S_2\neq 0$. By Lemma
\ref{lem:syzygyofGproj} $S_2$ is indecomposable. Then by above we
infer that $S_2$ is simple. Moreover we claim that a simple
$A$-module $S$ with ${\rm Ext}_A^1(S, S_2)\neq 0$ is isomorphic to
$S_1$.

To prove the claim, let us assume on the contrary that $S$ is not
isomorphic to $S_1$. Take a short exact sequence $0\rightarrow
K\rightarrow P\stackrel{\pi}\rightarrow S\rightarrow 0$ such that
$\pi$ is a projective cover. As above we infer that $K$ is
semisimple. Observe that ${\rm Ext}_A^1(S, S_2)\neq 0$ implies that
${\rm Hom}_A(K, S_2)\neq 0$. Then $S_2$ is a direct summand of $K$.
Thus we get a nonzero morphism $S_2\hookrightarrow K\hookrightarrow
P$ which is denoted by $l$. Note that ${\rm Ext}_A^1(S_1, P)=0$
since $S_1$ is Gorenstein-projective. By the long exact sequence
obtained by applying ${\rm Hom}_A(-, P)$ to $\eta$ we have an
epimorphism ${\rm Hom}_A(P_1, P)\rightarrow {\rm Hom}_A(S_2, P)$
induced by $i_2$. It follows then there exists a morphism $a\colon
P_1\rightarrow P$ such that $a\circ i_2=l$. Note that $S_2$ is the
socle of $P_1$ on which $a$ is nonzero. It follows that the morphism
$a$ is monic. On the other hand, since $S$ is not isomorphic to $S_1$,
the composite $P_1\stackrel{a}\rightarrow P\stackrel{\pi}\rightarrow
S$ is zero. This implies that the monomorphism $a$ factors through
$K$. Note that $K$ is semisimple while the module $P_1$ is not
semisimple. This is absurd. We are done with the claim.

Similarly we define $S_3$ by the short exact sequence $0\rightarrow
S_3\stackrel{i_3}\rightarrow P_2\stackrel{\pi_2}\rightarrow
S_2\rightarrow 0$ such that $\pi_2$ is a projective cover. As  above
$S_3$ is simple and satisfies that any simple $A$-module $S$ with
${\rm Ext}_A^1(S, S_3)\neq 0$ is isomorphic  to $S_2$. In this way
we define $S_n$ for $n\geq 1$.

Choose $n\geq 1$ minimal with the property that $S_n\simeq S_m$ for
 some $m<n$. Then such an $m$ is unique. Note that $m=1$. Otherwise
 ${\rm Ext}_A^1(S_{m-1}, S_n)\simeq {\rm Ext}_A^1(S_{m-1}, S_m)\neq 0$ while
 $S_{m-1}$ is not isomorphic to $S_{n-1}$. This contradicts the
 claim above for $S_n$. Then we get a set $\{S_1, S_2, \cdots, S_{n-1}\}$
 of pairwise non-isomorphic simple $A$-modules; moreover each $S_i$
 satisfies that any simple $A$-module $S$ with
${\rm Ext}_A^1(S, S_i)$ is isomorphic  to $S_{i-1}$, and clearly
from the construction of $S_i$'s  any simple $A$-module $S$ with
${\rm Ext}_A^1(S_i, S)$ is isomorphic  to $S_{i+1}$ (here we
identify $S_{0}$ with $S_{n-1}$, $S_{n}$ with $S_1$). It follows
then that the full sub quiver of $Q(A)$ with vertices $\{S_1, S_2,
\cdots, S_{n-1}\}$ is a connected component. Since $A$ is connected,
these are all the simple $A$-modules. Then all the indecomposable
projective $A$-modules are given by $\{P_1, P_2, \cdots, P_{n-1}\}$.
Observe that each of them  is of length $2$ and has a different
simple socle. It follows immediately that the algebra $A$ is
self-injective either by \cite[1.6, Ex.2]{Ben} or by \cite[Chapter
IV, Ex.12]{ARS}.
\end{proof}

\chapter{Gorenstein Homological Algebra}
In this chapter we will study the central topic of Gorenstein
homological algebra: we study proper Gorenstein-projective
resolutions and Gorenstein-projective dimensions; we study the class
of modules having finite Gorenstein-projective dimension; we also
briefly discuss Gorenstein derived categories.

\section{Gorenstein Resolutions and Dimensions}

In this section we will study proper Gorenstein-projective
resolutions of modules and then various Gorenstein dimensions of
modules and algebras.

Let $A$ be an artin $R$-algebra where $R$ is a commutative artinian
ring. Denote by $A\mbox{-GProj}$ the category of
Gorenstein-projective $A$-modules. Recall that the stable category
$A\mbox{-\underline{GProj}}$ modulo projective modules is
triangulated with arbitrary coproducts.

Following Neeman (\cite{Nee0} and \cite[Definition 1.7]{Nee1}) a
triangulated category $\mathcal{T}$ with arbitrary coproducts is
\emph{compactly generated}\index{category!compactly generated}
provided that the full subcategory $\mathcal{T}^c$ consisting of
compact objects is essentially small and for any nonzero object
$X\in\mathcal{T}$ there exists a compact object $C$ with ${\rm
Hom}_\mathcal{T}(C, X)\neq 0$. In this case the smallest
triangulated subcategory of $\mathcal{T}$ which contains
$\mathcal{T}^c$ and is closed under coproducts is $\mathcal{T}$
itself; see \cite[Lemma 3.2]{Nee1}. One of the main features of
compactly generated triangulated categories is that the Brown
representability theorem and its dual hold form them; see
\cite[Theorem 3.1]{Nee1} and \cite{Nee98}; also see \cite[Theorems
8.3.3 and 8.6.1]{Nee2001}.

The following result, due to Beligiannis (\cite[Theorem 6.6]{Bel2}),
is one of the basic results in Gorenstein homological algebra.
Observe that it is contained in \cite[Theorem 5.4]{IK}.

\begin{thm}\label{thm:existenceofresolution} {\rm (Beligiannis)}
Let $A$ be an artin algebra. Then the triangulated category
$A\mbox{-\underline{\rm GProj}}$ is compactly generated. \hfill
$\square$
\end{thm}

We will sketch a proof of this theorem in Appendix B. Here we will
give an application of the theorem.

The following  application of Beligiannis's Theorem  is contained
in \cite[Theorem 2.11]{Jor2}. We will present a stronger result due
to Beligiannis-Reiten in Appendix B. \index{subcategory!contravariantly finite}

\begin{cor}\label{cor:GProjresolution}
Let $A$ be an artin algebra. Then the subcategory $A\mbox{-{\rm
GProj}}\subseteq A\mbox{-{\rm Mod}}$ is contravariantly finite.
\end{cor}

\begin{proof}
Let $M$ be an $A$-module. Consider the contravariant functor
$$\underline{\rm Hom}_A(-, M)\colon A\mbox{-\underline{\rm
GProj}}\rightarrow R\mbox{-Mod},$$ where $R\mbox{-Mod}$ denotes the
category of (left) $R$-modules. This functor sends coproducts to
products and by Proposition \ref{prop:cohomologicalfunctor} it is
cohomological. We apply Theorem \ref{thm:existenceofresolution} and
Brown representability theorem (\cite[Theorem 3.1]{Nee1}). There
exists a Gorenstein-projective module $G$ with an isomorphism
$\eta\colon \underline{\rm Hom}_A(-, G)\simeq \underline{\rm
Hom}_A(-, M)$. This yields a morphism $f\colon G\rightarrow M$. Take
an epimorphism $\pi\colon P\rightarrow M$ with $P$ projective.

We claim that $(f, \pi)\colon G\oplus P\rightarrow M$ is a right
$A\mbox{-GProj}$-approximation of $M$. In fact, given a morphism
$g\colon G'\rightarrow M$ with $G'$ Gorenstein-projective, by the
isomorphism $\eta$ there exists a morphism $t\colon G'\rightarrow G$
such that $g-f\circ t$ becomes zero in the stable category, that is,
it factors through a projective module. Since $\pi$ is epic,
$g-f\circ t$ necessarily factors through $\pi$. Hence $g$ factors
though $(f, \pi)$.
\end{proof}

\begin{rem}
Recall from Proposition \ref{prop:GProj}(1) that the subcategory
$A\mbox{-{\rm GProj}}$ is closed under filtered colimits. Combining
the corollary above with \cite[Theorem 2.2.8]{Xu} we infer that the
subcategory $A\mbox{-{\rm GProj}}$ is covering in $A\mbox{-{\rm
Mod}}$. Using Wakamatsu's Lemma and this remark one deduces (part
of) Theorem \ref{thm:BR-AppendixB} directly.
\end{rem}

\begin{rem}
By a similar argument as above we prove that the subcategory
$A\mbox{-{\rm GProj}}$ is covariantly finite in $A\mbox{-{\rm
Mod}}$, and then $A\mbox{-{\rm GProj}}$ is functorially finite in
$A\mbox{-{\rm Mod}}$. For each $A$-module consider the covariant
functor $\underline{\rm Hom}_A(M, -)\colon A\mbox{-\underline{\rm
GProj}}\rightarrow R\mbox{-{\rm Mod}}$.  Then we apply the dual
Brown representability theorem to this functor to get a left
$A\mbox{-{\rm GProj}}$-approximation of $M$. Here one needs to use
the fact that the category $A\mbox{-{\rm Proj}}$ of projective
modules is covariantly finite in $A\mbox{-{\rm Mod}}$.
\end{rem}

Note that for an artin algebra $A$ the category $A\mbox{-Gproj}$ of
finitely generated Gorenstein-projective modules is  not necessarily
contravariantly finite in $A\mbox{-mod}$; see \cite{Yoh, BK}. While
Beligiannis's Theorem enables us to define Gorenstein extension
groups via resolutions by large Gorenstein-projective modules. This
is the main reason why we study Gorenstein homological algebra in
the category $A\mbox{-Mod}$ of large modules.

A complex $X^\bullet=(X^n, d^n_X)_{n\in \mathbb{Z}}$ of $A$-modules
is said to be \emph{right {\rm GP}-acyclic} \index{complex!right
GP-acyclic} provided that for each Gorenstein-projective module $G$
the Hom complex ${\rm Hom}_A(G, X^\bullet)$ is acyclic. A right
GP-acyclic complex is necessarily acyclic; moreover, an acyclic
complex $X^\bullet$ is right GP-acyclic if and only if each induced
morphism $X^n\rightarrow Z^{n+11}(X^\bullet)$  induces for each Gorenstein-projective
module $G$ a surjective map ${\rm Hom}_A(G, X^n)\rightarrow
{\rm Hom}_A(G, Z^{n+1}(X^\bullet))$.

Let $M$ be an $A$-module. By a \emph{Gorenstein-projective
resolution}, or a \emph{{\rm GP}-resolution}
\index{resolution!Gorenstein-projective, GP} in short, of $M$ we mean
an acyclic complex $\cdots \rightarrow G^{-2}\rightarrow
G^{-1}\rightarrow G^0\rightarrow M\rightarrow 0$ with each
$G^{-i}\in A\mbox{-GProj}$; sometimes we write this resolution as
$G^\bullet \rightarrow M$. A GP-resolution is \emph{proper}
\index{resolution!Gorenstein-projective, GP!proper} provided that in
addition it is right GP-acyclic (\cite[section 4]{AM}). It follows
from Corollary \ref{cor:GProjresolution} that each $A$-module admits
a proper GP-resolution. Such a proper GP-resolution is necessarily
unique.

The following two lemmas are well known.

\begin{lem}
{\rm (Comparison Theorem)} Let $M$ and $N$ be $A$-modules. Consider
two proper GP-resolutions $G_M^\bullet\rightarrow M$ and
$G_N^\bullet \rightarrow N$. Let $f\colon M\rightarrow N$ be a
morphism. Then there is a chain map $f^\bullet\colon
G_M^\bullet\rightarrow G_N^\bullet$ filling into the following
commutative diagram
\[\xymatrix{
\cdots \ar[r] & G_M^{-2} \ar@{.>}[d]^{f^{-2}} \ar[r] & G_M^{-1}
\ar@{.>}[d]^{f^{-1}} \ar[r] & G_M^0 \ar@{.>}[d]^{f^{0}}  \ar[r] & M
\ar[d]^{f}
\ar[r] & 0\\
\cdots \ar[r] & G_N^{-2} \ar[r] & G_N^{-1} \ar[r] & G_N^0 \ar[r] & N
\ar[r] & 0
 }\] Such a chain map is unique up to homotopy. \hfill $\square$
\end{lem}

\begin{lem}
{\rm (Horseshoe Lemma)} Let $0\rightarrow L \rightarrow M\rightarrow
N\rightarrow 0$ be a right GP-acyclic sequence. Take two proper
GP-resolutions $G_L^\bullet\rightarrow L$ and $G_N^\bullet
\rightarrow N$. Then there exists a commutative diagram
\[\xymatrix{
 L \ar[r] & M \ar[r]  & N \\
 G_L^0 \ar[u] \ar[r]^-{\binom{1}{0}} & G_L^0\oplus G_N^0
\ar@{.>}[u] \ar[r]^-{(0, 1)} & G_N^0 \ar[u]\\
 G_L^{-1} \ar[u] \ar[r]^-{\binom{1}{0}} & G_L^{-1}\oplus
G_N^{-1}
\ar@{.>}[u] \ar[r]^-{(0, 1)} & G_N^{-1} \ar[u]\\
 G_L^{-2} \ar[u] \ar[r]^-{\binom{1}{0}} & G_L^{-2}\oplus
G_N^{-2} \ar@{.>}[u] \ar[r]^-{(0, 1)} & G_N^{-2} \ar[u]\\
 \vdots \ar[u]& \vdots\ar@{.>}[u] & \vdots \ar[u] }\]
such that the middle column is a proper GP-resolution. \hfill
$\square$
\end{lem}

One of the central notions in Gorenstein homological algebra is
Gorenstein extension group defined below.

\begin{defn}
 Let $M$ and $N$ be $A$-modules. Let $n\geq 0$. Take a proper
GP-resolution $G^\bullet_M\rightarrow M$. Define the \emph{$n$-th
GP-extension group}\index{GP-extension group} of $N$ by $M$ to be
${\rm Ext}^n_{\rm GP}(M, N)=H^n({\rm Hom}_A(G^\bullet, N))$. We set
${\rm Ext}^{-n}_{\rm GP}(M, N)=0$ for $n\geq 1$.
\end{defn}

\begin{rem}
By Comparison Theorem the GP-extension groups do not depend on the
choice of the proper GP-resolution. As an immediate consequence we have
${\rm Ext}^0_{\rm GP}(M, N)={\rm Hom}_A(M, N)$, and ${\rm
Ext}^n_{\rm GP}(M, N)=0$ if $n\geq 1$ and $M$ is
Gorenstein-projective.
\end{rem}

Note that the $n$-th GP-extension groups ${\rm Ext}^n_{\rm GP}(M,
N)$ are functorial both in $M$ and $N$. Moreover we have the
following well-known results.

\begin{lem}\label{lem:LESTI}{\rm (Long Exact Sequence Theorem I)}
Let $0\rightarrow M'\rightarrow M\rightarrow M''\rightarrow 0$ be a
right GP-acyclic sequence and let $N$ be an $A$-module. Then there
is a long exact sequence
 \begin{align*}
0\rightarrow &{\rm Hom}_A(M'', N)
\rightarrow {\rm Hom}_A(M, N) \rightarrow {\rm Hom}_A(M,
N)\stackrel{c^0}\rightarrow {\rm Ext}^1_{\rm GP}(M'', N) \\
& \rightarrow {\rm Ext}^1_{\rm GP}(M, N) \rightarrow {\rm
Ext}^1_{\rm GP}(M', N) \stackrel{c^1}\rightarrow {\rm Ext}^2_{\rm
GP}(M'', N) \rightarrow {\rm Ext}^2_{\rm GP}(M, N) \rightarrow
\cdots
\end{align*}
where the morphisms $c^i$ are the \emph{connecting morphisms} and
the other morphisms are induced by the corresponding functors.
\hfill $\square$
\end{lem}

\begin{lem}\label{lem:LESTII}{\rm (Long Exact Sequence Theorem II)}
Let $0\rightarrow N'\rightarrow N\rightarrow N''\rightarrow 0$ be a
right GP-acyclic sequence and let $M$ be an $A$-module. Then there
is a long exact sequence
 \begin{align*}
0\rightarrow &{\rm Hom}_A(M, N') \rightarrow {\rm Hom}_A(M, N)
\rightarrow {\rm Hom}_A(M,
N'')\stackrel{c^0}\rightarrow {\rm Ext}^1_{\rm GP}(M, N') \\
& \rightarrow {\rm Ext}^1_{\rm GP}(M, N) \rightarrow {\rm
Ext}^1_{\rm GP}(M, N'') \stackrel{c^1}\rightarrow {\rm Ext}^2_{\rm
GP}(M, N') \rightarrow {\rm Ext}^2_{\rm GP}(M, N) \rightarrow \cdots
\end{align*}
where the morphisms $c^i$ are the \emph{connecting morphisms} and
the other morphisms are induced by the corresponding functors.
\hfill $\square$
\end{lem}

One of the main reasons to study the GP-extension groups is  that
they provide certain numerical invariants for modules and algebras.

\begin{defn}
For an $A$-module $M$ we define its \emph{Gorenstein-projective
dimension} \index{dimension!Gorenstein-projective} by ${\rm
GP.dim}\; M={\rm sup}\{n\geq 0\; |\; {\rm Ext}^n_{\rm GP}(M, -)\neq
0\}$. The \emph{global Gorenstein-projective dimension}
\index{dimension!Gorenstein-projective!global} of the algebra $A$,
denoted by ${\rm gl.GP.dim} \;A$, is defined to be the supreme of
the Gorenstein-projective dimensions of all modules. The large
(resp. small) \emph{finistic Gorenstein-projective dimension}
\index{dimension!Gorenstein-projective!(large/small) finistic} of
the algebra $A$, denoted by ${\rm Fin.GP.dim}\; A$ (resp. ${\rm
fin.Gp.dim}\; A$), is defined  to be the supreme of the
Gorenstein-projective dimensions of all (resp. finitely generated)
$A$-modules of finite Gorenstein-projective dimension.
\end{defn}

\begin{rem}
It follows from the definitions that ${\rm fin.Gp.dim}\; A\leq {\rm
Fin.GP.dim}\; A\leq {\rm gl.GP.dim}\; A$; if ${\rm gl.GP.dim}\; A<
\infty$ or ${\rm Fin.GP.dim}\; A=\infty$, then $ {\rm Fin.GP.dim}\;
A={\rm gl.GP.dim}\; A$  \hfill $\square$
\end{rem}

The following result is basic.

\begin{prop}\label{prop:GPdimension} Let $M$ be an $A$-module and let $n\geq 0$. Then
the following statements are equivalent:
\begin{enumerate}
\item[(1)] ${\rm GP.dim}\; M\leq n$;
\item[(2)] ${\rm Ext}_{\rm GP}^{n+1}(M, -)=0$;
\item[(3)] for each right GP-acyclic complex $0\rightarrow K\rightarrow G^{1-n}\rightarrow \cdots \rightarrow G^0\rightarrow M
\rightarrow 0$ with each $G^i$ Gorenstein-projective we have that
$K$ is Gorenstein-projective.
\end{enumerate}
\end{prop}

\begin{proof}
The implications ``$(1)\Rightarrow (2)$" and ``$(3)\Rightarrow (1)$"
are trivial. To see ``$(2)\Rightarrow (3)$", assume that ${\rm
Ext}_{\rm GP}^{n+1}(M, -)=0$ and that we are given  a  right
GP-acyclic complex $0\rightarrow K\rightarrow
G^{1-n}\rightarrow \cdots \rightarrow G^0\rightarrow M \rightarrow
0$ with each $G^i$ Gorenstein-projective. Take a right GP-acyclic sequence
$0\rightarrow K'\stackrel{j}\rightarrow G^{-n}\rightarrow K\rightarrow 0$
with $G^{-n}$ Gorenstein-projective; see Corollary \ref{cor:GProjresolution}. Denote by $K^{-i}$ the
image of $G^{-i}\rightarrow G^{1-i}$;  we identify $K^{-n}$ with $K$, $K^0$ with $M$.
Note that each sequence
$0\rightarrow K^{-i}\rightarrow G^{1-i}\rightarrow
K^{1-i}\rightarrow 0$ is right GP-acyclic. Here we identify $K^{-n-1}$ with $K'$.
By Lemma \ref{lem:LESTI} we can apply
dimension-shift to the these sequences. Then we get ${\rm Ext}_{\rm
GP}^{1}(K^{-n}, K')\simeq {\rm Ext}_{\rm GP}^{n+1}(M, K')=0$. By
Lemma \ref{lem:LESTI} again this implies that the induced morphism
${\rm Hom}_A(G^{-n}, K')\rightarrow {\rm Hom}_A(K', K')$ by $j$ is
epic and then the monomorphism $j$ splits. Hence $K$ is a direct summand of $G^{-n}$ and then it is
Gorenstein-projective.
\end{proof}

Denote by $(A\mbox{-GProj})^\perp$ the full subcategory of
$A\mbox{-Mod}$ consisting of modules $M$ with the property that
${\rm Ext}_A^n(G, M)=0$ for all $n\geq 1$ and $G$
Gorenstein-projective.

The following observation is of interest.

\begin{lem}\label{lem:GPorthogonal}
Let $M$ be an $A$-module. The following statements are equivalent:
\begin{enumerate}
\item[(1)] $M\in (A\mbox{-{\rm GProj}})^\perp$;
\item[(2)] $\underline{\rm Hom}_A(G, M)=0$ for all
Gorenstein-projective modules $G$;
\item[(3)] any epimorphism $P\rightarrow M$ with $P$ projective is a
right $A\mbox{-{\rm GProj}}$-approximation.
\end{enumerate}
\end{lem}

\begin{proof}
Note that the full subcategory $A\mbox{-GProj}$ of $A\mbox{-Mod}$ is
closed under syzygies and that every Gorenstein-projective module is
a syzygy module. Then ``$(1)\Leftrightarrow (2)$" follows from Lemma
\ref{lem:basicpropertyofsyzygy}. Note that (2) just means that any
morphism from a Gorenstein-projective module to $M$ factors through
a projective module, and then factors though any fixed epimorphism
$P\rightarrow M$ with $P$ projective. Then ``$(2)\Leftrightarrow
(3)$" follows immediately.
\end{proof}

We observe the following result.

\begin{cor}\label{cor:GPorthogonal}
Let $M$ be an $A$-module. Then $M\in (A\mbox{-{\rm GProj}})^\perp$
if and only if $\Omega M\in(A\mbox{-{\rm GProj}})^\perp$.
\end{cor}

\begin{proof}
Note that projective modules lie in $(A\mbox{-{\rm GProj}})^\perp$
and by dimension-shift the full subcategory $(A\mbox{-{\rm
GProj}})^\perp$ is closed under taking cokernels of monomorphisms. Then the
``if" part follows.

For the ``only if" part, assume that $M\in (A\mbox{-{\rm
GProj}})^\perp$. Take a Gorenstein-projective module $G$ and
consider the short exact sequence $\xi\colon 0\rightarrow \Omega
M\rightarrow P\stackrel{f}\rightarrow M\rightarrow 0$ with $P$
projective. By Lemma \ref{lem:GPorthogonal}(3) the morphism $f$ is a
right $A\mbox{-GProj}$-approximation. Then from the long exact
sequence obtained by applying ${\rm Hom}_A(G, -)$ to $\xi$ one
deduces that for all $n\geq 1$, ${\rm Ext}_A^n(G, \Omega M)$=0.
\end{proof}

We have the following comparison between GP-extension groups and the
usual extension groups.

\begin{prop}\label{prop:ExtGPExt} Let $M\in (A\mbox{-{\rm GProj}})^\perp$. Then any
projective resolution of $M$ is a proper GP-resolution.
Consequently, we have natural isomorphisms
$${\rm Ext}_{\rm GP}^n(M,
N)\simeq {\rm Ext}_A^n(M, N)$$
for all $n\geq 0$ and all modules
$N$.
\end{prop}

\begin{proof}
Consider a projective resolution $P^\bullet \rightarrow M$. By an
iterated application of  Corollary \ref{cor:GPorthogonal} we infer
that all the cocycles of $P^\bullet$ lie in $(A\mbox{-{\rm
GProj}})^\perp$. Then applying Lemma \ref{lem:GPorthogonal}(3)
repeatedly we infer that the projective resolution is a proper
GP-resolution.
\end{proof}

 For an artin algebra $A$ denote by ${\rm Fin.dim}\; A$ (resp. ${\rm fin.dim}\; A$) the
 large (resp. small)
 \emph{finistic dimension}\index{dimension!(large/small) finistic} of $A$. Observe that modules of finite projective
dimension lie in $(A\mbox{-{\rm GProj}})^\perp$; see Corollary
\ref{cor:infinitedimension}. Then we have the following immediate
consequence of the proposition above.

\begin{cor}\label{cor:Findim}
Let $A$ be an artin algebra. Then we have ${\rm Fin.dim}\; A\leq
{\rm Fin.GP.dim}\; A$ and ${\rm fin.dim}\; A\leq {\rm fin.Gp.dim}\;
A$. \hfill $\square$
\end{cor}

Let us remark that the results and the arguments in this section
carry over to Gorenstein-injective modules without any difficulty.
In particular we define the GI-extension groups ${\rm Ext}^n_{\rm
GI}(M, N)$ by using the \emph{proper Gorenstein-injective
coresolution} \index{proper Gorenstein-injective
coresolution} of the module $N$.

\section{Modules of Finite Gorenstein Dimension}

In this section we study the class of modules having finite
Gorenstein-projective dimension.

Let $A$ be an artin algebra. Recall that for an $A$-module $M$ its
Gorenstein-projective dimension is denote by ${\rm GP.dim}\; M$. For
each $n\geq 0$ denote by ${\rm GP}^{\leq n}(A)$ the full subcategory
of $A\mbox{-Mod}$ consisting of modules $M$ with ${\rm GP.dim}\;
M\leq n$, and denote by ${\rm GP}^{<\infty}(A)$ the full subcategory
of $A\mbox{-Mod}$ consisting of modules of finite
Gorenstein-projective dimension. Observe that ${\rm GP}^{\leq
0}(A)=A\mbox{-GProj}$.

The following result is basic; also see \cite[Theorem 2.10]{Hol}.

\begin{lem} \label{lem:AusBuch} {\rm (Auslander-Buchweitz)}
Let $M$ be an $A$-module which fits into an exact sequence
$0\rightarrow G^{-n}\rightarrow \cdots \rightarrow G^{-1}\rightarrow
G^0\rightarrow M\rightarrow 0$ with each $G^{-i}$
Gorenstein-projective. Then there is a proper GP-resolution
$0\rightarrow P^{-n}\rightarrow \cdots \rightarrow P^{-2}\rightarrow
P^{-1}\rightarrow G\rightarrow M\rightarrow 0$ such that $G$ is
Gorenstein-projective and each $P^{-i}$ is projective. In
particular, we have ${\rm GP.dim}\; M\leq n$.
\end{lem}

\begin{proof}
The existence of the second exact sequence follows from
\cite[Theorem 1.1]{ABu}. We will show that the sequence
$0\rightarrow P^{-n}\stackrel{d^{-n}}\rightarrow \cdots \rightarrow
P^{-2}\stackrel{d^{-2}}\rightarrow
P^{-1}\stackrel{d^{-1}}\rightarrow
G\stackrel{\varepsilon}\rightarrow M\rightarrow 0$ is a proper
GP-resolution. Note that for each $1\leq i\leq n$ the module ${\rm
Im}\; d^{-i}$ has finite projective dimension, and hence it lies in
$(A\mbox{-GProj})^\perp$. By Lemma \ref{lem:GPorthogonal}(3) each
morphism $P^{-i}\rightarrow {\rm Im}\;d^{-i}$ is a right
$A\mbox{-GProj}$-approximation. Since ${\rm Ker}\; \varepsilon={\rm
Im}\; d^{-1}$ lies in $(A\mbox{-GProj})^\perp$, it follows
immediately that $\varepsilon$ is a right
$A\mbox{-GProj}$-approximation. From these we conclude that the
sequence is a proper GP-resolution.
\end{proof}

The following result is contained in \cite[Proposition 2.1]{ABu};
also see \cite[Theorem 2.20]{Hol}.

\begin{prop}\label{prop:GPdimensionoffiniteGPdimension}
Let $M$ be an $A$-module of finite Gorenstein-projective dimension
and let $n\geq 0$. The following statements are equivalent:
\begin{enumerate}
\item[(1)] ${\rm GP.dim}\; M \leq n$;
\item[(2)] ${\rm Ext}_A^i(M, L)=0$ for $i\geq n+1$ and $L$ of finite projective dimension;
\item[(3)] ${\rm Ext}_A^i(M,A)=0$ for $i\geq n+1$;
\item[(4)] ${\rm Ext}_A^{n+1}(M, L)=0$ for $L$ of finite projective
dimension.
\end{enumerate}
\end{prop}

\begin{proof}
For ``$(1)\Rightarrow (2)$", take an exact sequence $0\rightarrow
G^{-n}\rightarrow \cdots \rightarrow G^{-1}\rightarrow
G^0\rightarrow M\rightarrow 0$ with each $G^{-i}$
Gorenstein-projective. Note that ${\rm Ext}_A^i(G^{-j}, L)=0$ for
all $i\geq 1$ and $0\leq j\leq n$. By dimension-shift we have that
for $i\geq n+1$, ${\rm Ext}_A^i(M, L)\simeq {\rm
Ext}^{i-1}_A(M^{-1}, L)\simeq {\rm Ext}^{i-2}_A(M^{-2}, L) \simeq
\cdots \simeq {\rm Ext}_A^{i-n}(G^{-n}, L)=0$, where each $M^{-i}$
is the image of $G^{-i}\rightarrow G^{1-i}$.

The implication ``$(2)\Rightarrow (3)$" is trivial. For
``$(3)\Rightarrow (4)$", first note that ${\rm Ext}_A^i(M,P)=0$ for
$i\geq n+1$  and all projective modules $P$. Then $(4)$ follows by
applying dimension-shift to a projective resolution of $L$.

To see ``$(4)\Rightarrow (1)$", we apply the lemma above to get a
long exact sequence  $0\rightarrow K \rightarrow
P^{-n}\stackrel{d^{-n}}\rightarrow \cdots \rightarrow
P^{-1}\rightarrow G\rightarrow M\rightarrow 0$ such that $G$ is
Gorenstein-projective, each $P^{-j}$ is projective and $K$ has
finite projective dimension. Note that ${\rm Ext}_A^i(G, K)=0$ for
all $i\geq 1$. We can apply dimension-shift to get that ${\rm
Ext}_A^1({\rm Im}\; d^{-n}, K)\simeq {\rm Ext}_A^{n+1}(M, K)=0$.
Consequently, the monomorphism $K\rightarrow P^{-n}$ splits. Then we
get a GP-resolution of $M$ of length $n$; see Lemma
\ref{lem:AusBuch}.
\end{proof}

The following important result is contained in \cite[section
3]{ABu}; compare \cite[Theorem 2.24]{Hol}.

\begin{prop}\label{prop:AusBuch} {\rm (Auslander-Buchweitz)}
Let $A$ be an artin algebra. If any two of the three terms in a
short exact sequence of $A$-modules have finite
Gorenstein-projective dimension, then so does the remaining term.
Moreover, a direct summand of a module of finite
Gorenstein-projective dimension is of finite Gorenstein-projective
dimension. \hfill $\square$
\end{prop}

We get the following result by applying  Propositions
\ref{prop:AusBuch} and \ref{prop:GPdimensionoffiniteGPdimension};
compare \cite[Proposition 2.18]{Hol}.

\begin{cor}\label{cor:GP.dim}
Let $0\rightarrow L\rightarrow M\rightarrow N\rightarrow 0$ be a
short exact sequence of $A$-modules. Then we have ${\rm GP.dim}\; M
\leq {\rm max}\{{\rm GP.dim}\; L, {\rm GP.dim}\; N\}$, ${\rm
GP.dim}\; N\leq {\rm max}\{{\rm GP.dim}\;L+1, {\rm GP.dim}\; M\}$
and ${\rm GP.dim}\; L \leq {\rm max}\{{\rm GP.dim}\; M, {\rm
GP.dim}\; N-1\}$.\hfill $\square$
\end{cor}

The following result is of interest; compare \cite[Theorem
2.20]{Hol}.

\begin{thm}\label{thm:boundofGP.dim}
Let $M$ be an $A$-module and let $n\geq 0$. Then the following
statements are equivalent:
\begin{enumerate}
\item[(1)] ${\rm GP.dim}\; M \leq n$;
\item[(2)] there exists an exact sequence $0\rightarrow G^{-n}\rightarrow
\cdots \rightarrow G^{-1}\rightarrow G^0\rightarrow M\rightarrow 0$
with each $G^{-i}$ Gorenstein-projective;
\item[(3)] for each exact sequence  $0\rightarrow K\rightarrow G^{1-n}
\rightarrow \cdots \rightarrow G^{-1}\rightarrow G^0\rightarrow
M\rightarrow 0$ with each $G^{-i}$ Gorenstein-projective we have
that $K$ is Gorenstein-projective.
\end{enumerate}
\end{thm}

\begin{proof}
We apply Lemma \ref{lem:AusBuch}. Then the equivalence
``$(1)\Leftrightarrow (2)$"  follows directly. The implication
``$(3)\Rightarrow (2)$" is trivial. To see ``$(1)\Rightarrow (3)$",
first note by applying Proposition \ref{prop:AusBuch} repeatedly we
get that the module $K$ has finite Gorenstein-projective dimension.
We apply dimension-shift to the given exact sequence. We get that
${\rm Ext}_A^i(K, A)\simeq {\rm Ext}_A^{i+n}(M, A)=0$ for $i\geq 1$.
By Proposition \ref{prop:GPdimensionoffiniteGPdimension}  we get
${\rm GP.dim}\; K=0$, that is, $K$ is Gorenstein-projective.
\end{proof}

The following immediate consequence of Theorem
\ref{thm:boundofGP.dim} is contained implicitly in \cite[Proposition 4.2]{AR1991};
also see\cite[Proposition 3.10]{Bel2}.

\begin{cor}
Let $A$ be an artin algebra. Then $A$ is Gorenstein if and only if
${\rm gl.GP.dim}\; A<\infty$. In this case we have ${\rm G.dim}\;
A={\rm gl.GP.dim}\; A$.
\end{cor}

\begin{proof}
We apply Theorem \ref{thm:boundofGP.dim} and Theorem
\ref{thm:GorensteinalgebraII}.
\end{proof}

The following result is due to Holm; compare \cite[Theorem
2.28]{Hol}.

\begin{thm} {\rm (Holm)}
Let $A$ be an artin algebra. Then we have ${\rm Fin.dim}\; A={\rm
Fin.GP.dim}\; A$ and ${\rm fin.dim}\; A={\rm fin.Gp.dim}\; A$.
\end{thm}

\begin{proof}
We only show the first equality and the second is proved similarly.
We have observed that ${\rm Fin.dim}\; A \leq {\rm Fin.GP.dim}\; A$
in Corollary \ref{cor:Findim}. By Lemma \ref{lem:AusBuch} we infer
that ${\rm Fin.GP.dim}\; A\leq {\rm Fin.dim}\; A+1$. Hence if ${\rm
Fin.GP.dim}\; A$ is infinite we are done.

Now assume that ${\rm Fin.GP.dim}\; A=m$ such that $0< m< \infty$.
Take a module $M$ with ${\rm GP.dim}\; M=m$. By Lemma
\ref{lem:AusBuch} there is a short exact sequence $0\rightarrow
K\rightarrow G\rightarrow M\rightarrow 0$ such that $G$ is
Gorenstein-projective and ${\rm proj.dim}\; K=m-1$. Take a short exact sequence
$0\rightarrow G\rightarrow P\rightarrow G'\rightarrow 0$ with $P$
projective and $G'$ Gorenstein-projective. Hence we get two short
exact sequences
$$0\longrightarrow K\longrightarrow P\longrightarrow L\longrightarrow 0 \mbox{  and  }
0\longrightarrow M\longrightarrow L\longrightarrow G'\longrightarrow
0.$$ Since $M$ is not Gorenstein-projective, by the second exact sequence
we have that $L$ is not Gorenstein-projective; see Proposition
\ref{prop:Gproj}. In particular, it is not projective. By the first
exact sequence we have ${\rm proj.dim}\; L=m$. We are done.
\end{proof}

The following observation is rather easy.

\begin{prop}
Let $A$ be an artin algebra. Then we have $${\rm gl.GP.dim}\; A={\rm
sup}\{ {\rm GP.dim}\; M|\; M\in A\mbox{-{\rm mod}}\}.$$
\end{prop}

\begin{proof}
Choose a complete set of representatives of pairwise non-isomorphic
simple $A$-modules  $\{S_1, \cdots, S_n\}$. Note that each
$A$-module has a finite filtration with semisimple factors. Using
the fact that ${\rm GP.dim}\; \oplus_{i} M_i={\rm sup}\{{\rm
GP.dim}\; M_i\}$. We apply Corollary \ref{cor:GP.dim} repeatedly to
infer that ${\rm GP.dim}\; M\leq {\rm max}\{ {\rm GP.dim}\; S_1,
\cdots, {\rm GP.dim}\; S_n\}$  for all $A$-modules $M$ .
\end{proof}

We end this section with a discussion on a certain balanced property
of Gorenstein extension groups.

The following observation is contained in the proof of \cite[Lemma
3.4]{Hol2}.

\begin{lem}\label{lem:perpGInj}
Let $0\rightarrow L\rightarrow M\rightarrow N\rightarrow 0$ be a
short exact sequence with $L\in {^\perp(A\mbox{-{\rm GInj}})}$ and
let $X$ be a Gorenstein-projective module. Then the following
induced sequence
$$ 0\longrightarrow {\rm Hom}_A(N, X) \longrightarrow {\rm Hom}_A(M, X)
 \longrightarrow {\rm Hom}_A(L, X) \longrightarrow 0$$
is exact.
\end{lem}

\begin{proof}
Denote the morphism $L\rightarrow M$ by $f$. It suffices to show
that for each morphism $a\colon L\rightarrow X$ there exists a
morphism $b\colon M\rightarrow X$ such that $b\circ f=a$. Since
$L\in {^\perp(A\mbox{-GInj})}$, the morphism $a$ factors through an
injective module $I$, say there are morphisms $a'\colon N\rightarrow
I$ and $i\colon I\rightarrow X$ such that $i\circ a'=a$; compare
Lemma \ref{lem:GPorthogonal}. By the injectivity of $I$ there is a
morphism $b'\colon M\rightarrow I$ with $b'\circ f=a'$. Set
$b=i\circ b'$.
\end{proof}

The following balanced property of Gorenstein extension groups is
due to Holm; see \cite[Theorem 3.6]{Hol2}.

\begin{thm}{\rm (Holm)} Let $A$ be an artin algebra. Let $M$ and $N$ be
$A$-modules with finite Gorenstein-projective (resp.
Gorenstein-injective) dimension. Then for each $n\geq 0$ there is an
isomorphism
$${\rm Ext}_{\rm GP}^n(M, N)\simeq {\rm Ext}_{\rm GI}^n(M, N)$$
which is functorial in both $M$ and $N$.
\end{thm}

\begin{proof}
Take a proper GP-resolution $0\rightarrow P^{-n}\rightarrow \cdots
\rightarrow P^{-1}\rightarrow G^0 \rightarrow M \rightarrow 0$ with
$G$ Gorenstein-projective and each $P^{-i}$ projective. Write it as
$G^\bullet\rightarrow M$. Note that all the cocycles of $G^\bullet$
(but $M$) have finite projective dimension and then lie in $^\perp
(A\mbox{-GInj})$. Let $X$ be a Gorenstein-injective $A$-module.
Applying Lemma \ref{lem:perpGInj} repeatedly we get that the induced
sequence ${\rm Hom}_A(M, X)\rightarrow {\rm Hom}_A(G^\bullet, X)$ is
acyclic.

Take a proper GI-coresolution $0\rightarrow N\rightarrow R
\rightarrow I^{1}\rightarrow \cdots \rightarrow I^m\rightarrow 0 $
with $R$ Gorenstein-injective and each $I^i$ injective. Write it as
$N\rightarrow R^\bullet$. Similarly as above we have that for each
Gorenstein-projective module $G$ the induced sequence ${\rm
Hom}_A(G, N)\rightarrow {\rm Hom}_A(G, R^\bullet)$ is acyclic. Now
consider the Hom bicomplex ${\rm Hom}_A(G^\bullet, R^\bullet)$ and
the associated two spectral sequences. The two spectral sequences
collapse to ${\rm Ext}^*_{\rm GP}(M, N)$ and ${\rm Ext}^*_{\rm
GI}(M, N)$, respectively; for details, consult \cite[Proposition
2.3]{EJ0}. Then we are done.
\end{proof}

Recall that an artin algebra $A$ is said to be \emph{virtually
Gorenstein} provided that $(A\mbox{-GProj})^\perp ={^\perp
(A\mbox{-GInj})}$;  see \cite{BR, Bel2}.

We observe the following characterization of virtually Gorenstein
algebras; see \cite[Chapter X, Theorem 3.4(v)]{BR}. \index{algebra!virtually Gorenstein}

\begin{prop}
Let $A$ be an artin algebra. Then $A$ is virtually Gorenstein if and
only if for all modules $M$ and $N$ and $n\geq 0$, there are
isomorphisms $${\rm Ext}^n_{\rm GP}(M, N)\simeq {\rm Ext}^n_{\rm
GI}(M, N),$$ which are functorial in both $M$ and $N$.
\end{prop}

\begin{proof}
For the ``if" part,  we apply Proposition \ref{prop:ExtGPExt}.
Observe that for $M\in (A\mbox{-GProj})^\perp$ and $R\in
A\mbox{-GInj}$ we have ${\rm Ext}^n_A(M, R)\simeq {\rm Ext}^n_{\rm
GP}(M, R)\simeq {\rm Ext}^n_{\rm GI}(M, R)=0$ for $n\geq 1$. This
shows that $(A\mbox{-GProj})^\perp\subseteq
{^\perp(A\mbox{-GInj})}$. Dually one shows that
$^\perp(A\mbox{-GInj})\subseteq (A\mbox{-GProj})^\perp$, and then
$A$ is virtually Gorenstein.

For the ``only if" part, we apply Corollary
\ref{cor:properGP-resAppendixB} (and its dual). Then the same proof
as in the theorem above works.
\end{proof}

\section{Gorenstein Derived Categories}

In this section we will briefly study Gorenstein derived categories
of an artin algebra $A$. The GP-extension and GI-extension groups of
two $A$-modules are encoded as the Hom spaces between certain
objects in the Gorenstein derived categories.

Let $A$ be an artin algebra. Denote by $\mathbf{K}(A\mbox{-Mod})$
the \emph{homotopy category}\index{category!homotopy} of complexes
in $A\mbox{-Mod}$. For a complex $X^\bullet=(X^n, d_X^n)_{n\in
\mathbb{Z}}$  its \emph{shifted complex} \index{complex!shifted}
$X^\bullet[1]$ is defined such that $(X^\bullet[1])^n=X^{n+1}$ and
$d_{X^\bullet[1]}^n=-d^{n+1}_X$. This gives rise to an automorphism
$[1]\colon \mathbf{K}(A\mbox{-Mod})\rightarrow
\mathbf{K}(A\mbox{-Mod})$. We denote by $[n]$ the $n$-th power of
$[1]$ for $n\in \mathbb{Z}$. A module $M$ is usually identified with
the \emph{stalk complex} $\cdots \rightarrow 0\rightarrow
M\rightarrow 0\rightarrow \cdots$ concentrated at degree zero. Then
for each $n$ the stalk complex $M[n]$ has $M$ at degree $-n$ and
zero elsewhere.\index{complex!stalk}

For a chain map $f^\bullet \colon X^\bullet \rightarrow Y^\bullet$
its \emph{mapping cone} \index{mapping cone} ${\rm Cone}(f^\bullet)$ is complex defined
such that for each $n\in \mathbb{Z}$
$${\rm
Cone}(f^\bullet)^n =Y^n\oplus X^{n+1} \mbox{ and } d_{{\rm
Cone}(f^\bullet)}^n=\begin{pmatrix} d_Y^{n} & f^{n+1} \\ 0 &
-d^n_{X}\end{pmatrix},$$ where $d_X^n$ and $d_Y^n$ are differentials
of $X^\bullet$ and $Y^\bullet$, respectively. The homotopy category
$\mathbf{K}(A\mbox{-Mod})$ has a canonical triangulated structure
such that all exact triangles are isomorphic to the \emph{standard
triangles} \index{standard triangle} $X^\bullet\stackrel{f^\bullet}\rightarrow
Y^\bullet\stackrel{\binom{1}{0}}\rightarrow {\rm
Cone}(f^\bullet)\stackrel{(0\; 1)}\rightarrow X^\bullet[1]$
associated to chain maps $f^\bullet$; for details, see \cite{V,Hap,
KZ}.

For an additive subcategory $\mathcal{X}$ of $A\mbox{-Mod}$ denote
by $\mathbf{K}^{-}(\mathcal{X})$ (resp.
$\mathbf{K}^{+}(\mathcal{X})$, $\mathbf{K}^{b}(\mathcal{X})$) the
full subcategory of $\mathbf{K}(A\mbox{-Mod})$ consisting of bounded
above (resp. bounded below, bounded) complexes in $\mathcal{X}$;
they are triangulated subcategories.

 We call a chain map $f^\bullet\colon X^\bullet \rightarrow
Y^\bullet$  a \emph{right GP-quasi-isomorphism} \index{right
GP-quasi-isomorphism} provided that for each Gorenstein projective
module $G$ the induced chain map ${\rm Hom}_A(G, f^\bullet)$ is a
quasi-isomorphism. Observe that a complex $X^\bullet$ is right GP-acyclic if
and only if the trivial map $X^\bullet\rightarrow 0$ is a right
GP-quasi-isomorphism. Moreover, a chain map $f^\bullet$ is a right
GP-quasi-isomorphism if and only if its mapping cone ${\rm
Cone}(f^\bullet)$ is right GP-acyclic. Denote by ${\rm
GP}\mbox{-ac}$ the full triangulated subcategory of
$\mathbf{K}(A\mbox{-Mod})$ consisting of right GP-acyclic complexes;
it is \emph{thick}, that is, the subcategory ${\rm GP}\mbox{-ac}$ is
closed under taking direct summands. Denote by $\Sigma_{\rm GP}$ the
class of all the right GP-quasi-isomorphisms in
$\mathbf{K}(A\mbox{-Mod})$; it is a saturated multiplicative system.

The following is initiated by Gao and Zhang \cite{GZ}; also see
\cite{Ch3}.

\begin{defn}
The \emph{Gorenstein-projective derived category} \index{category!Gorenstein-projective derived} $\mathbf{D}_{\rm
GP}(A)$ of an artin algebra $A$ is defined to be the Verdier
quotient category \index{category!Verdier quotient}
$$\mathbf{D}_{\rm GP}(A)\colon=\mathbf{K}(A\mbox{-{\rm Mod}})/{\rm GP}\mbox{-{\rm ac}}
=\mathbf{K}(A\mbox{-{\rm Mod}})[\Sigma_{\rm GP}^{-1}].$$ We denote
by $Q\colon \mathbf{K}(A\mbox{-{\rm Mod}})\rightarrow
\mathbf{D}_{\rm GP}(A)$ the quotient functor.
\end{defn}

Observe that  a chain map $f^\bullet\colon X^\bullet \rightarrow
Y^\bullet$ is a right GP-quasi-isomorphism if and only if
$Q(f^\bullet)$ is an isomorphism in $\mathbf{D}_{\rm GP}(A)$.

Dually one defines the \emph{Gorenstein-injective derived category}
\index{category!Gorenstein-injective derived} of $A$ to be
$\mathbf{D}_{\rm GI}(A)\colon=\mathbf{K}(A\mbox{-{\rm Mod}})/{\rm
GI}\mbox{-{\rm ac}}$, where ${\rm GI}\mbox{-{\rm ac}}$ is the full
triangulated subcategory of $\mathbf{K}\mbox{-Mod}$ consisting of
left GI-acyclic complexes. Both $\mathbf{D}_{\rm GP}(A)$ and
$\mathbf{D}_{\rm GI}(A)$ are called \emph{Gorenstein derived
categories} \index{category!Gorenstein derived} of $A$. In what
follows we will mainly consider the Gorenstein-projective derived
category.

The following result is basic.

\begin{lem}\label{lem:nullhomotopy}
Let $X^\bullet\in \mathbf{K}^{-}(A\mbox{-{\rm GProj}})$ and
$Y^\bullet\in {\rm GP}\mbox{-{\rm ac}}$. Then we have $${\rm
Hom}_{\mathbf{K}(A\mbox{-}{\rm Mod})}(X^\bullet, Y^\bullet)=0.$$
\end{lem}

\begin{proof}
Take a chain map $f^\bullet \colon X^\bullet \rightarrow Y^\bullet$.
Without loss of generality we assume that $X^n=0$ for $n>0$. Then we
have $d_Y^0\circ f^0=0$. Note that by assumption the Hom complex
${\rm Hom}_A(X^0, Y^\bullet)$ is acyclic. This implies that there
exists $h^0\colon X^0\rightarrow Y^{-1}$ such that
$f^0=d_Y^{-1}\circ h^0$. Set $h^n=0$ for $n\geq 1$.

 We make induction on $i\geq 0$. Assume that the morphisms
$h^{n}\colon X^n\rightarrow Y^{n-1}$ are defined for $n\geq -i$ such
that $f^n=d_Y^{n-1}\circ h^n+h^{n+1}\circ d_X^n$ for all $n\geq -i$.
We will construct $h^{-i-1}$. Note that
\begin{align*}
d_Y^{-i-1}\circ (f^{-i-1}-h^{-i}\circ d_X^{-i-1})&=f^{-i}\circ
d_X^{-i-1}- d_Y^{-i-1}\circ h^{-i}\circ
d_X^{-i-1}\\
&=h^{-i+1}\circ d_X^{-i}\circ d_X^{-i-1}=0.
\end{align*}
By assumption the Hom complex ${\rm Hom}_A(X^{-i-1}, Y^\bullet)$ is
acyclic. It follows that there exists $h^{-i-1}\colon
X^{-i-1}\rightarrow Y^{-i-2}$ such that $f^{-i-1}-h^{-i}\circ
d_X^{-i-1}=d_Y^{-i-2}\circ h^{-i-1}$. Continuing this argument we
find a homotopy $\{h^n\}_{n\in \mathbb{Z}}$ of the chain map
$f^\bullet$.
\end{proof}

We have the following direct consequence.

\begin{cor}
Let $X^\bullet\in \mathbf{K}(A\mbox{-{\rm Mod}})$ and $Y^\bullet \in
\mathbf{K}^{-}(A\mbox{-{\rm Mod}}) \cap {\rm GP}\mbox{-{\rm ac}}$.
Then we have
$${\rm
Hom}_{\mathbf{K}(A\mbox{-}{\rm Mod})}(X^\bullet, Y^\bullet)=0.$$
\hfill $\square$
\end{cor}

The following  consequence will be crucial to us.

\begin{cor}\label{cor:GPderived}
Let $X^\bullet\in \mathbf{K}^{-}(A\mbox{-{\rm GProj}})$ and
$Y^\bullet\in \mathbf{K}(A\mbox{-{\rm Mod}})$. Then the natural map
$${\rm
Hom}_{\mathbf{K}(A\mbox{-}{\rm Mod})}(X^\bullet,
Y^\bullet)\rightarrow {\rm Hom}_{\mathbf{D}_{\rm GP}(A)}(X^\bullet,
Y^\bullet)$$ sending $f^\bullet$ to $Q(f^\bullet)$ is an
isomorphism. In particular, the composite
$\mathbf{K}^{-}(A\mbox{-{\rm GProj}})\hookrightarrow
\mathbf{K}(A\mbox{-{\rm Mod}}) \stackrel{Q}\rightarrow
\mathbf{D}_{\rm GP}(A)$ is fully faithful.
\end{cor}

\begin{proof}
We apply Lemma \ref{lem:nullhomotopy}. Then this result is an
immediate consequence of \cite[\S 2, 5-3 Proposition]{V}.
\end{proof}

The following observation highlights Gorenstein derived categories;
see \cite{GZ}.

\begin{thm}
Let $M, N$ be $A$-modules and let $n\in \mathbb{Z}$. Then there is a
natural isomorphism
$${\rm Hom}_{\mathbf{D}_{\rm GP}(A)}(M, N[n])\simeq {\rm Ext}_{\rm GP}^n(M, N).$$
\end{thm}

\begin{proof}
Take a proper GP-resolution $\varepsilon\colon G^\bullet \rightarrow
M$. View $M$ as a stalk complex concentrated in degree zero and
$G^\bullet$ as a complex belonging to
$\mathbf{K}^{-}(A\mbox{-GProj})$. Note that $\varepsilon$ is a right
GP-quasi-isomorphism. Then $G^\bullet$ is isomorphic to $M$ in
$\mathbf{D}_{\rm GP}(A)$. We apply  Corollary \ref{cor:GPderived}.
Then we have
\begin{align*}{\rm Hom}_{\mathbf{D}_{\rm GP}(A)}(M,
N[n]) &\simeq {\rm Hom}_{\mathbf{D}_{\rm GP}(A)}(G^\bullet,
N[n])\\
&\simeq  {\rm Hom}_{\mathbf{K}(A\mbox{-}{\rm Mod})}(G^\bullet,
N[n])\\
&\simeq H^n({\rm Hom}_A(G^\bullet, N)).
\end{align*} By definition we have ${\rm Ext}_{\rm GP}^n(M, N)=H^n({\rm Hom}_A(G^\bullet,
N))$. We are done.
\end{proof}

We will finish this section with a remark on Gorenstein derived
categories. For more, we refer to \cite{GZ} and \cite{Ch3}.

Consider $\mathcal{E}_{\rm GP}$ the class of short exact sequence of
$A$-modules on which each functor ${\rm Hom}_A(G, -)$ is exact for
$G\in A\mbox{-GProj}$. Then the pair $(A\mbox{-Mod},
\mathcal{E}_{\rm GP})$ is an exact category in the sense of Quillen.
We will denote  this exact category by $A\mbox{-Mod}_{\rm GP}$.
Following Neeman (\cite[section 1]{Nee90}) a complex
$X^\bullet=(X^n, d_X^n)_{n\in \mathbb{Z}}$  is \emph{acyclic}
\index{complex!acyclic} in $\mathbf{K}(A\mbox{-Mod}_{\rm GP})$ if
and only if there are factorizations $d_X^n\colon
X^n\stackrel{p^n}\rightarrow Z^{n+1} \stackrel{i^{n+1}}\rightarrow
X^{n+1}$ such that for each $n$, $0\rightarrow Z^n
\stackrel{i^n}\rightarrow X^n\stackrel{p^n}\rightarrow
Z^{n+1}\rightarrow 0$ is a short exact sequence belonging to
$\mathcal{E}_{\rm GP}$. Observe that a complex is acyclic in
$\mathbf{K}(A\mbox{-Mod}_{\rm GP})$ if and only if it is right
GP-acyclic.

Following Neeman again (\cite[Remark 1.6]{Nee90}; also see
\cite[sections 11,12]{Ke96}) the \emph{derived category}
$\mathbf{D}(A\mbox{-Mod}_{\rm GP})$ \index{category!derived} of the
exact category $A\mbox{-Mod}_{\rm GP}$ is defined by
$$\mathbf{D}(A\mbox{-Mod}_{\rm GP})\colon =\mathbf{K}(A\mbox{-Mod}_{\rm GP})/{{\rm GP}\mbox{-ac}}. $$

Dually one may also consider the exact category $A\mbox{-Mod}_{\rm
GI}$ with the exact structure given by short exact sequences of
$A$-modules on which ${\rm Hom}_A(-, I)$ is exact for each $I\in
A\mbox{-GInj}$. Then one has the derived category
$\mathbf{D}(A\mbox{-Mod}_{\rm GI})$.

The last result can be viewed as a remark: roughly speaking,
Gorenstein derived categories are not ``new". This remark makes
possible to apply the general results on the derived categories of
exact categories to Gorenstein homological algebra; see \cite{Ke96}.

\begin{prop}
Let $A$ be an artin algebra. Then we have $\mathbf{D}_{\rm
GP}(A)=\mathbf{D}(A\mbox{-{\rm Mod}}_{\rm GP})$ and $\mathbf{D}_{\rm
GI}(A)=\mathbf{D}(A\mbox{-{\rm Mod}}_{\rm GI})$. \hfill $\square$
\end{prop}

\appendix

\chapter{Cotorsion Pairs}

In this section we review the theory of cotorsion pairs and other
relevant notions. The main references are \cite{Sa}, \cite[Chapter
7]{EJ} and \cite[Chapter 2]{GT}.

Throughout $\mathcal{A}$ is an abelian category. Let $\mathcal{X}$
be a full additive subcategory of $\mathcal{A}$. Let $M\in
\mathcal{A}$ be an object. A \emph{right
$\mathcal{X}$-approximation} \index{approximation!right} of $M$ is a
morphism $f\colon X\rightarrow M$ such that $X\in \mathcal{X}$ and
any morphism $X'\rightarrow M$ from an object $X'\in \mathcal{X}$
factors through $f$.  Dually one has the notion of \emph{left
$\mathcal{X}$-approximation}  \index{approximation!left}
(\cite{AS}). A right (resp. left) $\mathcal{X}$-approximation  is
also known as an $\mathcal{X}$-\emph{precover} \index{precover}
(resp. $\mathcal{X}$-\emph{preenvelop}) \index{preenvelop}
(\cite{Eno}). The subcategory $\mathcal{X}\subseteq \mathcal{A}$ is
said to be \emph{ contravariantly
finite}\index{subcategory!contravariantly finite}
\index{subcategory!covariantly finite}(resp. \emph{covariantly
finite}) provided that each object in $\mathcal{A}$ has a right
(resp. left) $\mathcal{X}$-approximation. The subcategory
$\mathcal{X}\subseteq \mathcal{A}$ is said to be \emph{functorially
finite} provided that it is both contravariantly finite and
covariantly finite.\index{subcategory!functorially finite}

 Let $\mathcal{X}\subseteq \mathcal{A}$ be a full additive subcategory. Denote by
$\mathcal{X}^{\perp_1}=\{Y\in \mathcal{A}\; |\; {\rm
Ext}_\mathcal{A}^1(X, Y)=0 \mbox{ for all } X\in \mathcal{X}\}$. A
\emph{special right $\mathcal{X}$-approximation} \index{approximation!right!special} of an object $M$ is
an  epimorphism $\phi\colon X\rightarrow M$ such that $X\in
\mathcal{X}$ and  the kernel ${\rm Ker}\; \phi$ lies in
$\mathcal{X}^{\perp_1}$. Observe that a special right
$\mathcal{X}$-approximation is a right $\mathcal{X}$-approximation.
Dually one has the notation $^{\perp_1}\mathcal{X}$ and the notion
of \emph{special left $\mathcal{X}$-approximation}. \index{approximation!left!special}

A right $\mathcal{X}$-approximation $f\colon X\rightarrow M$ is said
to be \emph{minimal} provided that any endomorphism $\theta\colon
X\rightarrow X$ with $f\circ \theta=f$ is necessarily an isomorphism
(\cite{AS}). \index{approximation!right!minimal} Such a minimal
right $\mathcal{X}$-approximation is also known as an
$\mathcal{X}$-\emph{cover}\index{cover}. Dually one has the notion
of $\mathcal{X}$-\emph{envelop} \index{envelop}(\cite{Eno}).

We  have the following useful lemma.\footnote{This lemma is supposed
to be found in \cite{Waka}, while I do not find it there. For a
proof, see \cite[Lemma 2.1.1]{Xu}; also see \cite{EJ} and
\cite[Lemma 2.1.13]{GT}.}

\begin{lem}{\rm (Wakamatsu's Lemma)} Let $\mathcal{X}\subseteq \mathcal{A}$ be a full
additive subcategory which is closed under extensions.
 Let $f\colon X\rightarrow M$ be an
$\mathcal{X}$-cover. Then ${\rm Ker}\; f$ lies in
$\mathcal{X}^{\perp_1}$. In particular, an epic $\mathcal{X}$-cover
is a special right  $\mathcal{X}$-approximation. \hfill $\square$
\end{lem}

The notion of cotorsion pair presented below is different from the
original one\footnote{The cotorsion pair introduced here is also
referred as a \emph{complete} cotorsion pair; see \cite{EJ, GT}.};
see \cite{Sa}. In our opinion  this one is more useful.

\begin{defn} {\rm (Salce)} A pair $(\mathcal{F},
\mathcal{C})$ of full additive subcategories in $\mathcal{A}$ is called a
\emph{cotorsion pair} \index{cotorsion pair} if the following conditions are satisfied:
\begin{enumerate}
\item[{\rm (C0)}] the subcategories $\mathcal{F}$ and $\mathcal{C}$
are closed under taking direct summands;
\item[{\rm (C1)}] ${\rm Ext}_\mathcal{A}^1(F, C)=0$ for all $F\in
\mathcal{F}$ and $C\in \mathcal{C}$;
\item[{\rm (C2)}] each object $M\in \mathcal{A}$ fits into a short
exact sequence $0\rightarrow C \rightarrow F\rightarrow M\rightarrow
0$ with $F\in \mathcal{F}$ and $C\in \mathcal{C}$;
\item[{\rm (C3)}] each object $M\in \mathcal{A}$ fits into a short
exact sequence $0\rightarrow M\rightarrow C'\rightarrow
F'\rightarrow 0$ with $C'\in \mathcal{C}$ and $F'\in \mathcal{F}$.
\hfill $\square$
\end{enumerate}
\end{defn}

\begin{rem}
(1). We have assumed that both $\mathcal{F}$ and $\mathcal{C}$ are
closed under taking direct summands. It follows immediately from the
conditions above that $\mathcal{F}={^{\perp_1}\mathcal{C}}$ and
$\mathcal{C}=\mathcal{F}^{\perp_1}$. In particular, one infers that
both $\mathcal{F}$ and $\mathcal{C}$ are closed under extensions.

(2). The condition (C2) claims that each object $M$ has a special
right $\mathcal{F}$-approximation. Hence the subcategory
$\mathcal{F}\subseteq \mathcal{A}$ is contravariantly finite. Dually
$\mathcal{C}$ is covariantly finite.
\end{rem}

Recall that for a full additive subcategory $\mathcal{X}$ in
$\mathcal{A}$ we denote by ${\rm fac}\; \mathcal{X}$ (resp. ${\rm
sub}\; \mathcal{X}$) the full subcategory of $\mathcal{A}$
consisting of factor objects (resp. sub objects) of objects in
$\mathcal{X}$.

We have the following result; compare \cite[Corollary 2.4]{Sa}. Let
us remark that using Wakamatsu's Lemma one deduces \cite[Proposition
1.9]{AR}  from the result below quite easily.

\begin{lem} {\rm (Salce's Lemma)}
Let $(\mathcal{F}, \mathcal{C})$ be a pair of full additive
subcategories in $\mathcal{A}$ such that
$\mathcal{F}={^{\perp_1}\mathcal{C}}$ and
$\mathcal{C}=\mathcal{F}^{\perp_1}$. Then the following statements
are equivalent:
\begin{enumerate}
\item[(1)] the pair $(\mathcal{F}, \mathcal{C})$ is a cotorsion
pair;
\item[(2)] ${\rm fac}\; \mathcal{F}=\mathcal{A}$ and the condition
(C3) holds;
\item[(3)] ${\rm sub}\; \mathcal{C}=\mathcal{A}$ and the condition
(C2) holds.
\end{enumerate}
\end{lem}

\begin{proof}
We will only show the equivalence ``$(1)\Leftrightarrow (2)$". The
implication ``$(1)\Rightarrow (2)$" is trivial. For the converse,
let $M\in \mathcal{A}$. Since ${\rm fac}\; \mathcal{F}=\mathcal{A}$
we may take a short exact sequence $0\rightarrow M'\rightarrow
F\rightarrow M\rightarrow 0$ with $F\in \mathcal{F}$. Applying the
condition (C3) to $M'$ we get a short exact sequence $0\rightarrow
M'\rightarrow C'\rightarrow F'\rightarrow 0$ with $C'\in
\mathcal{C}$ and $F'\in  \mathcal{F}$. Consider the following
pushout diagram.
\[\xymatrix{
& 0\ar[d] & 0\ar[d] \\
0\ar[r] &  M' \ar[d] \ar[r] & F \ar@{.>}[d] \ar[r] & M \ar@{=}[d]
\ar[r] & 0\\
0\ar[r] & C' \ar@{.>}[r] \ar[d] & E \ar@{.>}[r] \ar@{.>}[d] & M
\ar[r] & 0\\
& F' \ar[d]\ar@{=}[r] & F' \ar[d]\\
& 0 & 0 }\] By $\mathcal{F}={^{\perp_1}\mathcal{C}}$ the subcategory
$\mathcal{F}$ is closed under extensions. Consider the short exact
sequence in the middle column. We infer that $E\in \mathcal{F}$.
Then the short exact sequence in the middle row proves the condition
(C2) for $M$.
\end{proof}

We will recall a remarkable result due to Auslander and Buchweitz
on cotorsion pairs.

Let $\mathcal{X}\subseteq \mathcal{A}$ be a full additive
subcategory and let $n\geq 0$. Set $\mathcal{X}^n$ to the full
subcategory of $\mathcal{A}$ consisting of objects $M$ with an exact
sequence $0\rightarrow X^{-n}\rightarrow \cdots \rightarrow
X^{-1}\rightarrow X^0\rightarrow M\rightarrow 0$ such that each
$X^{-i}\in \mathcal{X}$. Note that $\mathcal{X}^0=\mathcal{X}$. Set
$\mathcal{X}^{-1}=0$. Denote by $\widehat{\mathcal{X}}$ the union of
all these $\mathcal{X}^n$'s.

Consider a full additive subcategory $\omega\subseteq \mathcal{X}$.
We say that $\omega$ \emph{cogenerates} $\mathcal{X}$ provided that
each object $X$ fits into a short exact sequence $0\rightarrow
X\rightarrow W\rightarrow X'\rightarrow 0$ with $W\in \omega$ and
$X'\in \mathcal{X}$. In this case $\omega$ is said to be a
\emph{cogenerator} of $\mathcal{X}$.\index{cogenerator}

The following is contained in \cite[Theorem 1.1]{ABu}. Let us remark
that it is proved directly by using induction on $n$ and taking
suitable pushout of short exact sequences.

\begin{prop} {\rm (Auslander-Buchweitz's decomposition theorem)}
Let $\mathcal{X}\subseteq \mathcal{A}$ be a full additive
subcategory which is closed under extensions. Let $\omega$ be a
cogenerator of $\mathcal{X}$ and let $n\geq 0$. Then for each $C\in
\mathcal{X}^n$, there are short exact sequences
\begin{align*}
&0\longrightarrow Y_C\longrightarrow X_C\longrightarrow  C\longrightarrow 0,\\
&0\longrightarrow  C \longrightarrow Y^C\longrightarrow  X^C
\longrightarrow 0,
\end{align*}
such that $X_C, X^C\in \mathcal{X}$, $Y^C\in \omega^{n-1}$ and
$Y_C\in \omega^n$. \hfill $\square$
\end{prop}

We say that a cogenerator $\omega$ of $\mathcal{X}$ is
\emph{Ext-injective} \index{cogenerator!Ext-injective} provided that ${\rm Ext}_\mathcal{A}^n(X, W)=0$ for
all $n\geq 1$, $X\in \mathcal{X}$ and $W\in \omega$. This implies by
dimension-shift that ${\rm Ext}_\mathcal{A}^1(X, C)=0$ for $X\in \mathcal{X}$ and $C\in
\widehat{\omega}$.

The proof of the following result is contained in the one of
\cite[Proposition 3.6]{ABu}.

\begin{thm}\label{thm:AB} {\rm (Auslander-Buchweitz)}
Let $\omega \subseteq \mathcal{X}\subseteq \mathcal{A}$ be full
additive subcategories such that $\omega$ is closed under taking
direct summands and $\mathcal{X}$ is closed under extensions and
taking direct summands. Suppose that
$\widehat{\mathcal{X}}=\mathcal{A}$ and that  $\omega$ is an
Ext-injective cogenerator of $\mathcal{X}$. Then $(\mathcal{X},
\widehat{\omega})$ is a cotorsion pair in $\mathcal{A}$.
\end{thm}

\begin{proof} We have observed that $\widehat{\omega}\subseteq
\mathcal{X}^{\perp_1}$. In view of the proposition above, it
suffices to show that $\widehat{\omega}$ is closed under taking
direct summands. In fact one shows that
$\widehat{\omega}={\mathcal{X}^{\perp_1}}$. It suffices to show that
$\widehat{\omega}\supseteq { \mathcal{X}^{\perp_1}}$. Let $C\in
{\mathcal{X}^{\perp_1}}$. Consider the short exact sequence
$0\rightarrow Y\rightarrow X \rightarrow C\rightarrow 0$ with $Y\in
\widehat{\omega}$ and $X\in \mathcal{X}$. Note that $Y\in
\widehat{\omega}\subseteq {\mathcal{X}^{\perp_1}}$ and then we have
$X\in {\mathcal{X}^{\perp_1} }$, since $\mathcal{X}^{\perp_1}$ is
closed under extensions. Note that $X$ fits into a short exact
sequence $0\rightarrow X\rightarrow W\rightarrow X'\rightarrow 0$
with $W\in \omega$ and $X'\in \mathcal{X}$. By $X\in
{\mathcal{X}^{\perp_1}}$ we infer that the sequence splits. Recall
that $\omega$ is closed under taking direct summands. We deduce that
$X\in \omega$ and then $C\in \widehat{\omega}$.
\end{proof}

\begin{rem}
Assume that $\mathcal{X}^n=\mathcal{A}$. Then we  have
$\widehat{\omega}=\omega^n$ by \cite[Proposition 3.6]{ABu}. Then the
cotorsion pair is given by $(\mathcal{X}, \omega^n)$.
\end{rem}

We will recall an important result which generates abundance of
cotorsion pairs. Recall that an abelian category $\mathcal{A}$
\emph{has enough projective objects} means that each object is a
factor object of a projective object.

The following important  result is contained in \cite[Theorem
2.4]{Hov}; compare  \cite[Theorem 10]{ET}.

\begin{thm}{\rm (Eklof-Trlifaj, Hovey)}
Let $\mathcal{A}$ be a Grothendieck category with enough projective
objects and let $\mathcal{S}\subseteq \mathcal{A}$ be a set of
objects. Set $\mathcal{C}=\mathcal{S}^{\perp_1}$ and
$\mathcal{F}={^{\perp_1}\mathcal{C}}$. Then $(\mathcal{F},
\mathcal{C})$ is a cotorsion pair in $\mathcal{A}$.\hfill $\square$
\end{thm}

The cotorsion pair above is called the cotorsion pair
\emph{cogenerated} by the set $\mathcal{S}$  of objects.\index{cotorsion pair!cogenerated by a set}

Finally we discuss resolving subcategories. Let $\mathcal{A}$ be an
abelian category with enough projective objects. A full additive
subcategory $\mathcal{X}$ is a \emph{resolving
subcategory}\index{subcategory!resolving} provided that it contains
all the projective objects and it is closed under extensions, taking
kernels of epimorphisms and direct summands (\cite[p.99]{ABr}). A
typical  example of a resolving subcategory is given by $^\perp
\mathcal{Y}$ where $\mathcal{Y}$ is a subcategory of $\mathcal{A}$
and $^\perp \mathcal{Y}=\{M\in \mathcal{A}\; |\; {\rm
Ext}_\mathcal{A}^i(M, Y)=0, \mbox{ for all } i\geq 1, Y\in
\mathcal{Y}\}$. Dually if $\mathcal{A}$ has enough injective
objects, we have the notion of \emph{coresolving
subcategory}\index{subcategory!coresolving} and a full subcategory
of the form $\mathcal{X}^\perp$ is coresolving.

The following observation is rather easy; see \cite[Lemma
2.2.10]{GT}.

\begin{prop}
Let $\mathcal{A}$ be an abelian category with enough projective and
injective objects. Let $(\mathcal{X}, \mathcal{Y})$ be a cotorsion
pair in $\mathcal{A}$. Then $\mathcal{X}$ is resolving if and only
if $\mathcal{Y}$ is coresolving. In this case we have
$\mathcal{X}={^\perp\mathcal{Y}}$ and
$\mathcal{Y}=\mathcal{X}^\perp$. \hfill $\square$
\end{prop}

In the case of this proposition the cotorsion pair $(\mathcal{X}, \mathcal{Y})$
is said to be \emph{hereditary}.

\chapter{A Proof of Beligiannis's Theorem}

In this section we will sketch a proof of Beligiannis's Theorem,
which claims that for an artin algebra  the stable category of
Gorenstein-projective modules modulo projective modules is a
compactly generated triangulated category; see Theorem
\ref{thm:existenceofresolution}. Our proof follows the one in
\cite{IK}. We will present a result due to Beligiannis and Reiten on
cotorsion pairs in the category of modules induced by the
subcategory of Gorenstein-projective modules.

For an additive category $\mathfrak{a}$ denote by
$\mathbf{K}(\mathfrak{a})$ the homotopy category
\index{category!homotopy} of complexes in $\mathfrak{a}$ and by
$\mathbf{K}^{+}(\mathfrak{a})$ (resp. $\mathbf{K}^{-}(\mathfrak{a})$
and $\mathbf{K}^b(\mathfrak{a})$) the subcategory consisting of
bounded below (resp. bounded above and bounded) complexes. Recall
that each of these homotopy categories has a canonical triangulated
structure.

Throughout $A$ will be an artin $R$-algebra where $R$ is a
commutative artinian ring. Denote by $A\mbox{-Mod}$ the category of
left $A$-modules and by  $A\mbox{-Proj}$ (resp. $A\mbox{-Inj}$) the
full subcategory consisting of projective (resp. injective)
$A$-modules. For an $A$-module $X$ denote by $DX={\rm Hom}_R(X, E)$
 its Matlis dual where $E$ is the minimal injective cogenerator
for $R$. Note that $DX$ is a right $A$-module and it is viewed as a
left $A^{\rm op}$-module. For an $A$-module $X$, set $X^*={\rm
Hom}_A(X, A)$ which has a natural left $A^{\rm op}$-module
structure.

Recall that a complex $P^\bullet$ of $A$-modules  is
\emph{homotopically projective} \index{complex!homotopically
projective} provided that any chain map from $P^\bullet$ to an
acyclic complex is homotopic to zero (\cite{Sp}); if in addition the
complex $P^\bullet$ consists of projective modules it is called
\emph{semi-projective}\index{complex!semi-projective}. For example,
a bounded above complex of projective modules is semi-projective.
For each complex $X^\bullet$ there is a \emph{semi-projective
resolution}\index{resolution!semi-projective}, that is, a
quasi-isomorphism $P^\bullet \rightarrow X^\bullet$ with $P^\bullet$
semi-projective; this semi-projective resolution is unique up to
homotopy. Denote  the complex $P^\bullet$ by $\mathbf{p}X^\bullet$.
This gives rise to a triangle functor $\mathbf{p}\colon
\mathbf{K}(A\mbox{-Mod})\rightarrow \mathbf{K}(A\mbox{-Proj})$ which
is called the \emph{semi-projective resolution
functor}\index{functor!semi-projective resolution}. For details, see
\cite[Chapter 8]{KZ}.

Dually we have the notions of \emph{homotopically
injective}\index{complex!homotopically injective} and
\emph{semi-injective}\index{complex!semi-injective} complexes. For
every complex $X^\bullet$ there is a unique semi-injective
resolution, that is, a quasi-isomorphism $X^\bullet\rightarrow
\mathbf{i}X^\bullet$ with $\mathbf{i}X^\bullet$ semi-injective. This
gives rise to the \emph{semi-injective resolution functor}
\index{functor!semi-injective resolution} $\mathbf{i}\colon
\mathbf{K}(A\mbox{-Mod})\rightarrow \mathbf{K}(A\mbox{-Inj})$; for
details, see \cite[Chapter 8]{KZ}.

We begin with the following fact.

\begin{lem}\label{lem:fact1}
Let $X^\bullet$ be a complex of $A$-modules. Then we have a natural
isomorphism $D\mathbf{p}X^\bullet \simeq \mathbf{i}DX^\bullet$; if
the complex $X^\bullet$ is bounded below, then we have $D\mathbf{i}
X^\bullet\simeq \mathbf{p}D X^\bullet$.
\end{lem}

\begin{proof}
Applying $D$ to the quasi-isomorphism
$\mathbf{p}X^\bullet\rightarrow X^\bullet$ we get a
quasi-isomorphism $DX^\bullet \rightarrow D\mathbf{p}X^\bullet$.
Note that the complex $D\mathbf{p}X^\bullet$ consists of injective
modules. To show the first isomorphism it suffices to show that
$D\mathbf{p}X^\bullet$ is homotopically injective. For an acyclic
complex $N^\bullet$, we have ${\rm Hom}_{\mathbf{K}(A\mbox{-} {\rm
Mod})}(N^\bullet, D\mathbf{p}X^\bullet)\simeq {\rm
Hom}_{\mathbf{K}(A\mbox{-}{\rm Mod})}(\mathbf{p}X^\bullet,
DN^\bullet)=0$ since $\mathbf{p}X^\bullet$ is homotopically
projective and $DN^\bullet$ is acyclic. The second isomorphism is
easy to prove.
\end{proof}

Recall that the \emph{Nakayama functor} \index{functor!Nakayama}
$\nu=DA\otimes_A\colon A\mbox{-Mod}\rightarrow A\mbox{-Mod}$ has a
right adjoint $\nu^-={\rm Hom}_A(DA, -) \colon
A\mbox{-Mod}\rightarrow A\mbox{-Mod}$. Note that the Nakayama
functor induces an equivalence $\nu\colon A\mbox{-Proj}\rightarrow
A\mbox{-Inj}$ whose quasi-inverse is given by $\nu^{-}$.

Applying the above to complexes, we have an equivalence $\nu\colon
\mathbf{K}(A\mbox{-Proj})\rightarrow \mathbf{K}(A\mbox{-Inj})$ with
quasi-inverse given by $\nu^{-}$.

Denote by  $A\mbox{-mod}$ (resp. $A\mbox{-proj}$, $A\mbox{-inj}$)
the category of finitely generated (resp. projective, injective)
$A$-modules. We note the following fact.

\begin{lem}\label{lem:fact2}
Let $X^\bullet\in \mathbf{K}^{-}(A\mbox{-{\rm mod}})$. Then we have
a natural isomorphism $\nu^{-}\mathbf{i}X^\bullet\simeq
(\mathbf{p}DX^\bullet)^*$.
\end{lem}

\begin{proof}
By Lemma \ref{lem:fact1} we have $\mathbf{p}DX^\bullet \simeq
D\mathbf{i}X^\bullet$. Observe that we may assume that
$\mathbf{i}X^\bullet$ lies in $\mathbf{K}^{-}(A\mbox{-inj})$. Note
that for a module $I\in A\mbox{-inj}$ we have a natural isomorphism
$\nu^{-} I\simeq (DI)^*$. From this we infer that
$(D\mathbf{i}X^\bullet)^*\simeq \nu^{-}\mathbf{i}X^\bullet$. We are
done.
\end{proof}

 Denote by $\mathbf{K}^{-,b}(A \mbox{-{\rm
proj}})$ the full subcategory  of $\mathbf{K}^{-}(A\mbox{-proj})$
consisting of complexes with only finitely many nonzero
cohomologies. For the notions of compact objects \index{compact
object} and compactly generated triangulated
categories\index{category!compactly generated}, we refer to
\cite{Nee0, Nee1}.

The following result is due to J{\o}rgensen (\cite[Theorem
2.4]{Jor1}); also see Krause (\cite[Example 2.6]{Kr}) and Neeman
(\cite[Propositions 7.12 and 7.14]{Nee2}).

\begin{lem}\label{lem:fact3}
Let $A$ be an artin algebra. Then the homotopy category
$\mathbf{K}(A\mbox{-{\rm Proj}})$ is compactly generated; moreover,
a complex is compact if and only if it is isomorphic to a complex of
the form $(P^\bullet)^*$ for $P^\bullet\in \mathbf{K}^{-,b}(A^{\rm
op}\mbox{-{\rm proj}})$. \hfill $\square$
\end{lem}

Denote by $\mathbf{K}_{\rm tac}(A\mbox{-Proj})$ the full subcategory
of $\mathbf{K}(A\mbox{-Proj})$ consisting of totally acyclic
complexes. It is a triangulated subcategory. Denote by
$A\mbox{-GProj}$ the category of Gorenstein-projective $A$-modules;
it is a Frobenius exact category such that its projective objects
are equal to projective $A$-modules. Denote by
$A\mbox{-\underline{GProj}}$ the \emph{stable category} \index{category!stable} of $A\mbox{-GProj}$
modulo projective modules; it has a canonical triangulated
structure. For details, see Chapter 2.

The following result is well known; see \cite[Lemma 7.3]{Kr}.

\begin{lem}\label{lem:fact4}
There is a triangle equivalence $A\mbox{-\underline{\rm
GProj}}\stackrel{\sim}\longrightarrow \mathbf{K}_{\rm
tac}(A\mbox{-{\rm Proj}})$ sending a Gorenstein-projective module to
its complete resolution, the quasi-inverse of which is given by the
functor $Z^0(-)$ of taking the zeroth cocycles. \hfill $\square$
\end{lem}

Recall that for a complex $X^\bullet$ of $A$-modules we denote by
$H^n(X^\bullet)$ its $n$-th cohomology for each $n\in \mathbb{Z}$.

\begin{lem}\label{lem:fact5}
Let $X^\bullet\in \mathbf{K}(A\mbox{-{\rm Mod}})$ and $I^\bullet \in
\mathbf{K}(A\mbox{-{\rm Inj}})$. For each $n\in \mathbb{Z}$, we have
the following natural isomorphisms
\begin{align*}
{\rm Hom}_{\mathbf{K}(A\mbox{-}{\rm Mod})}(A, X^\bullet[n])\simeq
H^n(X^\bullet) \mbox{ and } {\rm Hom}_{\mathbf{K}(A\mbox{-}{\rm
Mod})}(\mathbf{i}A, I^\bullet[n])\simeq H^n(I^\bullet).
\end{align*}
\end{lem}

\begin{proof}
The first isomorphism is well known and the second follows from the
first one and  \cite[Lemma 2.1]{Kr}.
\end{proof}

For a subset $\mathcal{S}$ of objects in a triangulated category
$\mathcal{T}$, consider its \emph{right orthogonal subcategory}
$\mathcal{S}^\perp=\{X\in \mathcal{T}\; |\; {\rm Hom}_\mathcal{T}(S,
X[n])=0, \mbox{ for all }n\in \mathbb{Z}, S\in \mathcal{S}\}$; \index{subcategory!right orthogonal} it is
a triangulated subcategory of $\mathcal{T}$.

We need  the following important result; see \cite[Theorem
2.1]{Nee0} and \cite[Proposition 1.7(1)]{IK}.

\begin{lem}\label{lem:fact6}
Let $\mathcal{T}$ be a compactly generated triangulated category and
let $\mathcal{S}$ be a set of compact objects in $\mathcal{T}$. Then
the right orthogonal subcategory $\mathcal{S}^\perp$ is compactly
generated. \hfill $\square$
\end{lem}

 Now we are in the position to prove Beligiannis's Theorem.

\begin{thm}{\rm (Beligiannis)} Let $A$ be an artin algebra. Then the
triangulated category $A\mbox{-\underline{\rm GProj}}$ is compactly
generated.
\end{thm}

\begin{proof} Recall that a complex $P^\bullet$ of projective $A$-modules is totally acyclic if
and only if $P^\bullet$ and $\nu P^\bullet$ are both acyclic; see
Lemma \ref{lem:totallyacyclic,large}. Note that
$$H^n(\nu P^\bullet)\simeq {\rm Hom}_{\mathbf{K}(A\mbox{-}{\rm
Inj})}(\mathbf{i}A, \nu P^\bullet[n])\simeq {\rm
Hom}_{\mathbf{K}(A\mbox{-}{\rm Proj})}(\nu^{-}\mathbf{i}A,
P^\bullet[n]),$$ where the first isomorphism is by Lemma
\ref{lem:fact5} and the second follows from the equivalence
$\nu\colon \mathbf{K}(A\mbox{-Proj})\stackrel{\sim}\longrightarrow
\mathbf{K}(A\mbox{-Inj})$. By Lemma \ref{lem:fact2} we have
$\nu^{-}\mathbf{i}A\simeq (\mathbf{p}DA)^*$ and by Lemma
\ref{lem:fact3} it is compact in $\mathbf{K}(A\mbox{-Proj})$. Then
it is direct to conclude that in $\mathbf{K}(A\mbox{-Proj})$ we have
$\mathbf{K}_{\rm tac}(A\mbox{-Proj})=\{A,
\nu^{-}\mathbf{i}(A)\}^\perp$. Then the result follows from Lemmas
\ref{lem:fact4} and \ref{lem:fact6}.
\end{proof}

\vskip 5pt

We note the following  consequence of Beligiannis's Theorem.

\begin{cor} \label{cor:trianglesfrombrwonrep}
Let $A$ be an artin algebra. Then each complex $P^\bullet$ of
projective $A$-modules fits into a triangle
$$P_1^\bullet\longrightarrow P^\bullet\longrightarrow P_2^\bullet\rightarrow P_1^\bullet[1]$$
such that $P_1^\bullet\in \mathbf{K}_{\rm tac}(A\mbox{-{\rm Proj}})$
and $P_2^\bullet\in \mathbf{K}_{\rm tac}(A\mbox{-{\rm
Proj}})^\perp$.
\end{cor}

\begin{proof}
Consider the inclusion functor $\mathbf{K}_{\rm tac}(A\mbox{-{\rm
Proj}})\hookrightarrow \mathbf{K}(A\mbox{-{\rm Proj}})$; it
preserves coproducts. By Beligiannis's Theorem and Lemma
\ref{lem:fact4} the category $\mathbf{K}_{\rm tac}(A\mbox{-{\rm
Proj}})$ is compactly generated. We apply Brown representability
theorem to get a right adjoint of this inclusion (\cite[Theorem
4.1]{Nee1}). The adjoint yields for each complex $P^\bullet$ such a
triangle; see \cite[Chapter 9]{Nee2}.
\end{proof}

For a class $\mathcal{S}$ of $A$-modules, set
$\mathcal{S}^\perp=\{X\in A\mbox{-Mod}\; |\; {\rm Ext}_A^i(S,
X)=0\mbox{ for all } i\geq 1, S\in \mathcal{S}\}$. The following
result is contained in \cite[Chapter X, Theorem 2.4(iv)]{BR}; also
see \cite[Proposition 3.4]{BK}. Observe that it is stronger than
Corollary \ref{cor:GProjresolution}.

\begin{thm}\label{thm:BR-AppendixB} {\rm (Beligiannis-Reiten)}
Let $A$ be an artin algebra. Then the pair $(A\mbox{-{\rm GProj}},
(A\mbox{-{\rm GProj}})^\perp)$ is a cotorsion pair in  $A\mbox{-{\rm
Mod}}$.
\end{thm}

\begin{proof}
Note that both $A\mbox{-{\rm GProj}}$ and $ (A\mbox{-{\rm
GProj}})^\perp$ are closed under taking direct summands. Then it
suffices to show that for an $A$-module $M$, there are short exact
sequences $0\rightarrow Y\rightarrow G\rightarrow M\rightarrow 0$
and $0\rightarrow M\rightarrow Y'\rightarrow G'\rightarrow 0$ such
that $G, G'\in A\mbox{-{\rm GProj}}$ and $Y, Y'\in (A\mbox{-{\rm
GProj}})^\perp$.

 We apply Corollary \ref{cor:trianglesfrombrwonrep}
to a projective resolution $\mathbf{p}M$ of $M$. We get a triangle
$P_1^\bullet\rightarrow  \mathbf{p}M\rightarrow P_2^\bullet
\rightarrow P_1^\bullet[1]$ with  $P_1^\bullet\in \mathbf{K}_{\rm
tac}(A\mbox{-{\rm Proj}})$ and $P_2^\bullet\in \mathbf{K}_{\rm
tac}(A\mbox{-{\rm Proj}})^\perp$. Note that $P_1^\bullet$ is acyclic
and then  $H^n(P_2^\bullet)=0$ for $n\neq 0$.

By rotating the triangle and adding some null-homotopic complexes to
$P_1^\bullet$ and $P_2^\bullet[1]$, we may assume that we have a
short exact sequence $0\rightarrow P_2^\bullet[-1]\rightarrow
P_1^\bullet\rightarrow \mathbf{p}M\rightarrow 0$ of complexes. For
each complex $X^\bullet$ denote by $C^0(X^\bullet)$ the cokernel of
$d^{-1}_X$. Applying $C^0(-)$ to the sequence, we get a short exact
sequence $0\rightarrow Y\rightarrow G\rightarrow M\rightarrow 0$ of
modules. Since $P_1^\bullet$ is totally acyclic,  the module $G$ is
Gorenstein-projective. We claim that $Y\in (A\mbox{-{\rm
GProj}})^\perp$. Note that the brutally truncated complex
$\sigma^{\leq 0}(P_2^\bullet[-1])$ is a projective resolution of
$Y$. We write $\mathbf{p}Y=\sigma^{\leq 0}(P_2^\bullet[-1])$. Take
$G'$ to be a Gorenstein-projective module and $P^\bullet$ its
complete resolution. Then we have the following isomorphisms
\begin{align*}
\underline{\rm Hom}_A(G, Y)&\simeq {\rm
Hom}_{\mathbf{K}(A\mbox{-}{\rm Mod})}(P^\bullet[-1], Y)\\
&\simeq {\rm Hom}_{\mathbf{K}(A\mbox{-}{\rm Proj})}(P^\bullet[-1],
\mathbf{p}Y)\\
&= {\rm Hom}_{\mathbf{K}(A\mbox{-}{\rm Proj})}(P^\bullet[-1], \sigma^{\leq 0}(P_2^\bullet[-1]))\\
&\simeq {\rm Hom}_{\mathbf{K}(A\mbox{-}{\rm Proj})}(P^\bullet[-1],
P_2^\bullet[-1])=0,
\end{align*}
where the first isomorphism is easy to see, the second follows from
the dual of \cite[Lemma 2.1]{Kr} and the fourth follows from that
fact that all chain morphisms from a totally acyclic complex to a
bounded below complex of projective modules is null-homotopic. Here
$\underline{\rm Hom}_A(-, -)$ means the morphism spaces in the
stable category $A\mbox{-\underline{Mod}}$ of $A\mbox{-Mod}$ modulo
projective modules. Then we are done with the claim by Lemma
\ref{lem:GPorthogonal}.

We have shown the first sequence. For the second on, we may assume
that there is a short exact sequence $0\rightarrow
\mathbf{p}M\rightarrow P_2^\bullet\rightarrow
P_1^\bullet[1]\rightarrow 0$ of complexes.
Similar as above we are done by
applying the functor $C^0(-)$ to this sequence.
\end{proof}

We end this section with an immediate consequence of
Beligiannis-Reiten's Theorem.

\begin{cor}\label{cor:properGP-resAppendixB}
Let $M$ be an $A$-module. Then there exists a proper GP-resolution
$$\cdots \longrightarrow P^{-2}\longrightarrow P^{-1}\longrightarrow G\longrightarrow M\longrightarrow 0$$
such that each $P^{-i}$ is projective and $G$ is
Gorenstein-projective.
\end{cor}

\begin{proof}
We apply Beligiannis-Reiten's Theorem and Lemma
\ref{lem:GPorthogonal}.
\end{proof}

\chapter{Open Problems}

In this section we will list some open problems in Gorenstein
homological algebra of artin algebras. They are mainly on CM-finite
artin algebras.

Let $A$ be an artin algebra. Recall that $A$ is
CM-finite\index{algebra!CM-finite} provided that up to isomorphism
there are only finitely many indecomposable finitely generated
Gorenstein-projective $A$-modules. Observe that algebras of finite
representation type are CM-finite. Recall that a remarkable result
due to Auslander states that an  artin algebra $A$ is of finite
representation type if and only if every (not necessarily finitely
generated) $A$-module is a direct sum of finitely generated ones.

The following analogue of Auslander's result for
Gorenstein-projective modules is asked in \cite{Ch1}.

\vskip 10pt

 \noindent {\bf Problem A.}\quad  \emph{Is it true that an artin
algebra $A$ is CM-finite if and only if every Gorenstein-projective
$A$-module is a direct sum of finitely generated ones?}

\vskip 10pt

 An affirmative answer is given for the case where $A$ is
Gorenstein\index{algebra!Gorenstein} (\cite{Ch1}). Recall that an
artin algebra $A$ is virtually Gorenstein\index{algebra!virtually
Gorenstein} if $(A\mbox{-GProj})^\perp={^\perp(A\mbox{-GInj})}$ and
that Gorenstein algebras are virtually Gorenstein. In fact, an
affirmative answer to Problem A is given even for the case where $A$
is virtually Gorenstein (\cite{Bel3}).

Based on the results in \cite{BK}, Problem A is equivalent to the
following one.

\vskip 10pt

 \noindent {\bf Problem B.}\quad \emph{ Is it true that a CM-finite artin
algebra $A$  is virtually Gorenstein?}

\vskip 10pt

Let us remark that an affirmative answer to Problem B is given in
\cite[Example 8.4(2)]{Bel2}, while the argument there is
incorrect\footnote{I would like to thank Professor Apostolos
Beligiannis for a private communication concerning this remark. The
argument in the second to last sentence in \cite[Example
8.4(2)]{Bel2} is incomplete.}. Hence Problem B stays open at
present.

Recall that an artin algebra $A$ is CM-free\index{algebra!CM-free} provided that
$A\mbox{-Gproj}=A\mbox{-proj}$. Closedly related to Problem A is the
following.

\vskip 10pt

 \noindent {\bf Problem C.}\quad  \emph{For a CM-free artin algebra $A$,
 do we have $A\mbox{-{\rm GProj}}=A\mbox{-{\rm Proj}}$?}

\vskip 10pt

For a CM-finite artin algebra $A$, take $G$ to be an additive
generator of $A\mbox{-Gproj}$. We call the algebra $\Gamma={\rm
End}_A(G)$ the \emph{CM-Auslander}\index{algebra!CM-Auslander}
algebra of $A$. Recall that there is, up to Morita equivalence, a
one-to-one correspondence between algebras of finite representation
type and algebras having global dimension at most $2$ and dominant
dimension at least $2$; this correspondence is called the
\emph{Auslander correspondence}.

\vskip 10pt

 \noindent {\bf Problem D.}\quad  \emph{What kinds of artin algebras is the CM-Auslander
 algebra of a CM-finite artin algebra? Is there an analogue of Auslander
 correspondence relating CM-finite artin algebras with their CM-Auslander
 algebras?}

\vskip 10pt

We call that an artin algebra $A$ is \emph{CM-bounded} \index{algebra!CM-bounded} provided that
the dimensions of all indecomposable finitely generated
Gorenstein-projective $A$-modules are uniformly bounded. Recall that
a famous theorem due to Roiter states that an artin algebra $A$ is
of finite representation type if the dimensions of all
indecomposable finitely generated $A$-modules are uniformly bounded;
see \cite{Roi}.

The following question then is natural.

\vskip 10pt

 \noindent {\bf Problem E.}\quad  \emph{Is a CM-bounded artin algebra necessarily CM-finite?}

\vskip 10pt

An affirmative answer to this problem is known in the case where $A$
is a $1$-Gorenstein algebra\index{algebra!Gornstein!$1$-Gorenstein}.

Recall that the stable category $A\mbox{-\underline{Gproj}}$ of
$A\mbox{-Gproj}$ modulo projective modules is a triangulated
category. However the information carried by this category is not
clear yet\footnote{I would like to thank Dr. Guodong Zhou for
discussions on this problem.}.

\vskip 10pt

 \noindent {\bf Problem F.}\quad  \emph{What is the Grothendieck group $K_0(A\mbox{-\underline{\rm Gproj}})$?
 What about other invariants of the algebra $A$ given by $A\mbox{-\underline{\rm Gproj}}$?}

\vskip 10pt

Recall that $\mathbf{D}_{\rm GP}(A)$ is the Gorenstein-projective
derived category of $A$\index{category!Gorenstein-projective
derived}. In our  point of view, the properties and the structure of
this category are far from clear.

\vskip 10pt

 \noindent {\bf Problem G.}\quad  \emph{Does $\mathbf{D}_{\rm GP}(A)$ always have arbitrary
 coproducts? For what kinds of algebras $A$ the category $\mathbf{D}_{\rm GP}(A)$  is compactly generated?}

\vskip 10pt

Note that for a self-injective algebra $A$ we have $\mathbf{D}_{\rm GP}(A)=\mathbf{K}(A\mbox{-Mod})$; in
this case, $\mathbf{D}_{\rm GP}(A)$ is compactly generated if and only if $A$ is of
finite representation type (\cite[Proposition 2.6]{Sto}). Moreover for a CM-finite  Gorenstein algebra $A$,
by combining the results in \cite{Ch1} and \cite{Ch3} one infers that
the category $\mathbf{D}_{\rm GP}(A)$ is compactly generated.

\printindex

\end{document}